\pdfoutput=1
\RequirePackage{ifpdf}
\ifpdf % We are running pdfTeX in pdf mode
\documentclass[pdftex]{sigma}
\else
\documentclass{sigma}
\fi

\numberwithin{equation}{section}
\newtheorem{thm}{Theorem}[section]
\newtheorem{lem}[thm]{Lemma}
\newtheorem{prop}[thm]{Proposition}
{\theoremstyle{definition}
\newtheorem*{rem}{Remark}}

\begin{document}

\allowdisplaybreaks

\newcommand{\arXivNumber}{1812.09791}

\renewcommand{\thefootnote}{}

\renewcommand{\PaperNumber}{075}

\FirstPageHeading

\ShortArticleName{Twisted de Rham Complex on Line and Singular Vectors in $\widehat{{\mathfrak{sl}_2}}$ Verma Modules }

\ArticleName{Twisted de Rham Complex on Line \\ and Singular Vectors in $\boldsymbol{\widehat{{\mathfrak{sl}_2}}}$ Verma Modules\footnote{This paper is a~contribution to the Special Issue on Algebra, Topology, and Dynamics in Interaction in honor of Dmitry Fuchs. The full collection is available at \href{https://www.emis.de/journals/SIGMA/Fuchs.html}{https://www.emis.de/journals/SIGMA/Fuchs.html}}}

\Author{Alexey SLINKIN~$^\dag$ and Alexander VARCHENKO~$^{\dag\ddag}$}

\AuthorNameForHeading{A.~Slinkin and A.~Varchenko}

\Address{$^\dag$~Department of Mathematics, University of North Carolina at Chapel Hill,\\
\hphantom{$^\dag$}~Chapel Hill, NC 27599-3250, USA}
\EmailD{\href{mailto:slinalex@live.unc.edu}{slinalex@live.unc.edu}, \href{mailto:anv@email.unc.edu}{anv@email.unc.edu}}
\URLaddressD{\url{http://varchenko.web.unc.edu/}}

\Address{$^\ddag$~Faculty of Mathematics and Mechanics, Lomonosov Moscow State University,\\
\hphantom{$^\ddag$}~Leninskiye Gory 1, 119991 Moscow GSP-1, Russia}

\ArticleDates{Received May 30, 2019, in final form September 21, 2019; Published online September 26, 2019}

\Abstract{We consider two complexes. The first complex is the twisted de Rham complex of scalar meromorphic differential forms on projective line, holomorphic on the complement to a finite set of points. The second complex is the chain complex of the Lie algebra of $\mathfrak{sl}_2$-valued algebraic functions on the same complement, with coefficients in a tensor product of contragradient Verma modules over the affine Lie algebra $\widehat{{\mathfrak{sl}_2}}$. In [Schechtman~V., Varchenko~A., \textit{Mosc. Math.~J.} \textbf{17} (2017), 787--802] a construction of a monomorphism of the first complex to the second was suggested and it was indicated that under this monomorphism the existence of singular vectors in the Verma modules (the Malikov--Feigin--Fuchs singular vectors) is reflected in the relations between the cohomology classes of the de Rham complex. In this paper we prove these results.}

\Keywords{twisted de Rham complex; logarithmic differential forms; $\widehat{{\mathfrak{sl}_2}}$-modules; Lie algebra chain complexes}

\Classification{17B56; 17B67; 33C80}

\begin{flushright}
\begin{minipage}{65mm}
\it Dedicated to Dmitry Borisovich Fuchs\\ on the occasion of his 80-th birthday
\end{minipage}
\end{flushright}

\renewcommand{\thefootnote}{\arabic{footnote}}
\setcounter{footnote}{0}

\section{Introduction}

We consider two complexes. The first complex is the twisted de Rham complex of scalar meromorphic differential forms on projective line, that are holomorphic on the complement to a finite set of points. The second complex is the chain complex of the Lie algebra of $\mathfrak{sl}_2$-valued algebraic functions on the same complement, with coefficients in a tensor product of contragradient Verma modules over the affine Lie algebra $\widehat{{\mathfrak{sl}_2}}$. In \cite{SV2} a construction of a monomorphism of the first complex to the second was suggested. That construction gives a relation between the singular vectors in the Verma modules and resonance relations in the de Rham complex.

That construction of the homomorphism was invented in the middle of 90s, while the paper~\cite{SV2} was prepared for publication 20 years later, when the proofs were forgotten,
if they existed. The paper~\cite{SV2} provides supporting evidence to the results formulated in~\cite{SV2}, but not the proofs. The goal of this paper is to give the proofs to the results formulated in~\cite{SV2}, namely, the proofs that the construction in~\cite{SV2} indeed gives a homorphism of complexes and relates the resonances in the de Rham complex and the~$\widehat{{\mathfrak{sl}_2}}$ singular vectors.

The construction in \cite{SV2} has two motivations.

The first motivation was to generalize the principal construction of~\cite{SV1}. In~\cite{SV1}, the tensor products of contragradient Verma modules over a semisimple Lie algebra were identified with the spaces of the top degree logarithmic differential forms over certain configuration spaces. Also the logarithmic parts of the de Rham complexes over the configuration spaces were identified with some standard Lie algebra chain complexes having coefficients in these tensor products, cf.\ in~\cite{KS, KV} a~$\mathcal D$-module explanation of this correspondence.

The second idea was that the appearance of {singular vectors} in Verma modules over affine Lie algebras is reflected in the relations between the cohomology classes of logarithmic differential forms. This was proved in an important particular case in~\cite{FSV,FSV+}, and in~\cite{STV} a one-to-one correspondence was established ``on the level of parameters''. In \cite{SV2} and in the present paper this correspondence is developed for another non-trivial class of singular vectors, namely for (a~part of) {Malikov--Feigin--Fuchs} singular vectors, cf.~\cite{MFF}.

The paper has the following structure. In Section \ref{sec DR} we introduce the de Rham complex of a~master function and resonance relations. In Section~\ref{sec hsl} we discuss $\widehat{{\mathfrak{sl}_2}}$ Verma modules, the Kac--Kazhdan reducibility conditions. We formulate Theorem~\ref{thm form} which describes certain relations in a contragradient Verma module. The proof of Theorem~\ref{thm form} is the main new result of this paper. In Theorem~\ref{limit} we describe the connection between the relations, described in Theorem~\ref{thm form}, and the Malikov--Feigin--Fuchs singular vectors. In Section~\ref{sec homs} we construct a map of the de Rham complex of the master function to the chain complex of the Lie algebra of $\mathfrak{sl}_2$-valued algebraic functions. Theorem~\ref{thm hom} says that the map is a monomorphism of complexes. The proof of Theorem~\ref{thm hom} is the second new result of this paper. Section~\ref{sec PROOF} is devoted to the proof of Theorem~\ref{thm form}. The proof is straightforward but rather nontrivial and lengthy.

\section{The de Rham complex of master function} \label{sec DR}

\subsection{Twisted de Rham complex} \label{def dr}
Consider ${\mathbb C}$ with coordinate $t$. Define the {\it master function} by the formula
\begin{gather*}%\label{Master}
\Phi(t) = \prod_{i=1}^n (t-z_i)^{-m_i/\kappa} ,
\end{gather*}
where $z_1, \dots, z_n, m_1,\dots,m_n, \kappa$ $\in{\mathbb C}$ are parameters. Fix these parameters and assume that $z_1,\dots,z_n$ are distinct. Set
\begin{gather*}
z_{n+1}=\infty, \qquad m_{n+1}=m_1+\dots+m_n-2.
\end{gather*}
Denote $U={\mathbb C}-\{z_1,\dots,z_n\}$.

Consider the {\em twisted de Rham complex} associated with $\Phi$,
\begin{gather}\label{dr com}
0\longrightarrow \Omega^0(U)\overset{\partial}{\longrightarrow}\Omega^1(U) \longrightarrow 0.
\end{gather}
Here $\Omega^p(U)$ is the space of rational differential $p$-forms on ${\mathbb C}$ regular on $U$. The differential $\partial$ is given by the formula
\begin{gather}\label{nabla}
\partial={\rm d}+\alpha\wedge\cdot ,
\end{gather}
where ${\rm d}$ is the standard de Rham differential and the second summand is the left exterior multiplication by the form
\begin{gather*}%\label{alpha}
\alpha=-\frac{1}{\kappa}\sum_{i=1}^n m_i \frac{{\rm d}t}{t-z_i} =\frac{{\rm d}\Phi}\Phi.
\end{gather*}
Formula (\ref{nabla}) is motivated by the computation
\begin{gather*}
{\rm d}(\Phi\omega)=\Phi {\rm d}\omega+{\rm d}\Phi\wedge\omega=\Phi({\rm d}\omega+\alpha\wedge\omega).
\end{gather*}
The complex $\Omega^{\bullet}(U)$ is the complex of global {\em algebraic} sections of
the de Rham complex of $\big({\mathcal{O}}^{\rm an}_U,\partial\big)$, where $\partial={\rm d}+\alpha\wedge\cdot$ is considered as the integrable connection on the sheaf ${\mathcal{O}}^{\rm an}_U$ of holomorphic functions on $U$.

If the monodromy of $\Phi$ is non-trivial, that is, if at least one of the numbers $m_1/\kappa, \dots, m_n/\kappa$ is not an integer, then
\begin{gather*}
H^0(\Omega^{\bullet}(U))=0,\qquad \dim H^1(\Omega^{\bullet}(U))=n-1,
\end{gather*}
see for example \cite{STV}.

\subsection[Basis of $\Omega^{\bullet}(U)$]{Basis of $\boldsymbol{\Omega^{\bullet}(U)}$} \label{Basis of Omega}

The functions
\begin{gather*}
\frac1{(t-z_i)^{a}}\ \ \text{for}\ a\in{\mathbb Z}_{>0} \qquad\text{and}\qquad t^a\ \ \text{for}\ a\in{\mathbb Z}_{\geq 0}
\end{gather*}
form a basis of $\Omega^0(U)$. The differential forms
\begin{gather*}
\frac{{\rm d}t}{(t-z_i)^a}\ \ \text{for}\ a\in{\mathbb Z}_{>0} \qquad \text{and}\qquad t^a{\rm d}t\ \ \text{for}\ a\in{\mathbb Z}_{\geq 0}
\end{gather*}
form a basis of $\Omega^1(U)$. The differential $\partial$ is given by the formulas
\begin{gather}
\kappa \partial\left(\frac1{(t-z_i)^{a}}\right) = -(m_i+a\kappa)\frac{{\rm d} t}{(t-z_i)^{a+1}}
+ \sum_{k=1}^a\sum_{j\neq i} \frac{m_j}{(z_j-z_i)^k} \frac{{\rm d}t}{(t-z_i)^{a+1-k}} \nonumber\\
\hphantom{\kappa \partial\left(\frac1{(t-z_i)^{a}}\right) =}{} -\sum_{j\neq i}\frac{ m_j}{(z_j-z_i)^a} \frac{{\rm d}t}{t-z_j}, \label{bff} \\
\kappa \partial\big(t^a\big) = \left(a\kappa -\sum_{j=1}^n m_j\right)t^{a-1}dt -\sum_{k=1}^{a-1} \sum_{j=1}^n m_jz_j^k t^{a-1-k}{\rm d}t-\sum_{j=1}^n m_jz_j^a \frac{{\rm d}t}{t-z_j}.\label{bfi}
\end{gather}

\subsection{Resonances}\label{reson}
The equations
\begin{enumerate}\itemsep=0pt
\item[(i)] $m_i+(a-1)\kappa=0$ for some $a\in {\mathbb Z}_{>0}$, $i\in\{1,\dots,n\}$,
\item[(ii)] $m_{n+1}+2-a\kappa=0$ for some $a\in {\mathbb Z}_{>0}$,
\item[(iii)] $\kappa=0$,
\end{enumerate}
are called the {\it resonance relations} for the parameters $m_1,\dots,m_{n+1}$, $\kappa$ of the de Rham complex.

If $\kappa=0$, then the twisted de Rham complex is not defined. If the resonance relation \mbox{$m_i+a\kappa=0$} is satisfied for some $a$, then the first term in the right-hand side of (\ref{bff}) equals zero. Similarly, if the resonance relation $m_{n+1}+2-a\kappa=0$ is satisfied for some $a$, then the first term in the right-hand side of (\ref{bfi}) equals zero.

\subsection{Logarithmic subcomplex} \label{log forms}
Let $\Omega_{\log}^0(U) \subset \Omega^0(U)$ be the subspace generated over ${\mathbb C}$ by function $1$. Let $\Omega_{\log}^1(U) \subset \Omega^1(U)$ be the subspace generated over ${\mathbb C}$ by the differential forms
\begin{gather*}
\omega_j=\frac{{\rm d}t}{t-z_j}, \qquad j=1,\dots,n.
\end{gather*}
These subspaces form the {\it logarithmic} subcomplex $(\Omega^\bullet_{\log}(U), \partial)$ of the de Rham complex\linebreak $(\Omega^\bullet(U), \partial)$. We have
\begin{gather*}
\partial\colon \ 1 \mapsto \alpha.
\end{gather*}

For generic $m_1,\dots, m_n$, $\kappa$, the embedding $(\Omega^\bullet_{\log}(U), \partial)\hookrightarrow (\Omega^\bullet(U), \partial)$ is a quasi-isomorphism, the logarithmic forms $\omega_1,\dots,\omega_n$ generate the space $H^1(\Omega^\bullet(U))$, and the cohomological relation $\sum\limits_{i=1}^n m_i \omega_i\sim 0$ is the only one, see for example~\cite{STV}.

Each resonance relation implies a new cohomological relation between the forms $\omega_1,\dots,\omega_n$, see \cite[Corollary 6.4]{SV2}. For example, if $m_{n+1}+2-\kappa=0$, then $\sum\limits_{j=1}^n z_j m_j\omega_j \sim 0$, and if $m_{n+1}+2-2\kappa=0$, then
\begin{gather*}
\sum_{j=1}^n z_j^2m_j\omega_j-\frac{1}{\kappa}\left(\sum_{j=1}^n\ z_jm_j\right)\left(\sum_{i=1}^n\ z_im_i\omega_i\right)\sim 0.
\end{gather*}

\section[$\widehat{{\mathfrak{sl}_2}}$-modules]{$\boldsymbol{\widehat{{\mathfrak{sl}_2}}}$-modules} \label{sec hsl}
\subsection[Lie algebra $\widehat{{\mathfrak{sl}_2}}$]{Lie algebra $\boldsymbol{\widehat{{\mathfrak{sl}_2}}}$} \label{lie alg hsl}

Let $\mathfrak{sl}_2$ be the Lie algebra of complex $(2\times 2)$-matrices with zero trace. Let $e$, $f$, $h$ be standard generators subject to the relations
\begin{gather*}
[e,f]=h,\qquad [h,e]=2e,\qquad [h,f]=-2f.
\end{gather*}
Let $\widehat{{\mathfrak{sl}_2}}$ be the affine Lie algebra $\widehat{{\mathfrak{sl}_2}}=\mathfrak{sl}_2\big[T,T^{-1}\big]\oplus{\mathbb C} c$ with the bracket
\begin{gather*}
\big[aT^i,bT^j\big] = [a,b]T^{i+j} + i \langle a,b\rangle \delta_{i+j,0} c ,
\end{gather*}
where $c$ is central element, $\langle a,b\rangle=\operatorname{tr} (ab)$. Set
\begin{alignat*}{4}
&e_1 = e, \qquad && f_1 = f, \qquad && h_1=h,& \\
& e_2 = fT, \qquad && f_2=eT^{-1}, \qquad && h_2=c-h.&
\end{alignat*}
These are the standard Chevalley generators defining $\widehat{{\mathfrak{sl}_2}}$ as the Kac--Moody algebra correspon\-ding to the Cartan matrix
$\left(\begin{smallmatrix} \hphantom{-}2&-2 \\
-2&\hphantom{-}2 \end{smallmatrix}\right).$

\subsection[Automorphism $\pi$]{Automorphism $\boldsymbol{\pi}$} \label{aut}
The Lie algebra $\widehat{{\mathfrak{sl}_2}}$ has an automorphism $\pi$,
\begin{gather*}
\pi \colon \ c \mapsto c, \qquad eT^i \mapsto fT^i, \qquad fT^i\mapsto eT^i, \qquad hT^i\mapsto -hT^i.
\end{gather*}

\subsection{Verma modules} \label{Ver mod}
We fix $k\in{\mathbb C}$ and assume that the central element~$c$ acts on all our representations by multiplication by~$k$.

For $m\in{\mathbb C}$, let $V(m,k-m)$ be the $\widehat{{\mathfrak{sl}_2}}$ {\em Verma module} with generating vector $v$. The Verma module is generated by $v$ subject to the relations
\begin{gather*}
e_1v=0,\qquad e_2v=0,\qquad h_1v=mv,\qquad h_2v=(k-m)v.
\end{gather*}

Let $\hat{\mathfrak{n}}_-\subset\widehat{{\mathfrak{sl}_2}}$ be the Lie subalgebra generated by $f_1$, $f_2$ and $U\hat{\mathfrak{n}}_-$ its enveloping algebra. The map $U\hat{\mathfrak{n}}_-\to V(m,k-m)$, $F\mapsto Fv$, is an isomorphism of $U\hat{\mathfrak{n}}_-$-modules.

The space $V(m,k-m)$ has a ${\mathbb Z}^2_{\geq 0}$-grading: a vector $f_{i_1}\cdots f_{i_p}v$\ with $i_j\in\{1,2\}$ has deg\-ree~$(p_1,p_2)$, if $p_i$ is the number of $i$'s in the sequence $i_1,\dots,i_p$. For $\gamma\in{\mathbb Z}^2_{\geq 0}$, denote by $V(m,k-m)_{\gamma}\subset V(m,k-m)$ the corresponding $\gamma$-homogeneous component.

A homogeneous nonzero vector $\omega$ in $V(m,k-m)$, non-proportional to $v$, is called a {\em singular vector} if $e_1\omega=e_2\omega=0$. The Verma module $V(m,k-m)$ is reducible, if and only if it contains a singular vector.

\subsection{Reducibility conditions}\label{reduc}
See Kac--Kazhdan \cite{KK}. Set
\begin{gather*}
\kappa=k+2.
\end{gather*}
The Verma module $V(m,k-m)$ is reducible if and only if at least one of the following relations holds:
\begin{enumerate}\itemsep=0pt
\item[(a)] $m-l+1+(a-1)\kappa=0$,
\item[(b)] $m+l+1-a\kappa=0$,
\item[(c)] $\kappa=0$,
\end{enumerate}
where $l,a\in{\mathbb Z}_{>0}$. If $(m,\kappa)$ satisfies exactly one of the conditions~(a),~(b), then $V(m,k-m)$ contains a unique proper submodule, and this submodule is generated by a~singular vector of degree $(la,l(a-1))$ for condition~(a) and of degree $(l(a-1),la)$ for condition~(b).

These singular vectors are highly nontrivial and are given by the following theorem.

\begin{thm}[{Malikov--Feigin--Fuchs, \cite{MFF}}]\label{thm MFF} For $a, l\in{\mathbb Z}_{>0}$ and $\kappa\in {\mathbb C}$, the monomials
\begin{gather*}
F_{12}(l,a,\kappa) = f_1^{l+(a-1)\kappa}f_2^{l+(a-2)\kappa} f_1^{l+(a-3)\kappa}\cdots f_2^{l-(a-2)\kappa} f_1^{l-(a-1)\kappa}, \\ %\label{a} \\
F_{21}(l,a,\kappa) = f_2^{l+(a-1)\kappa}f_1^{l+(a-2)\kappa} f_2^{l+(a-3)\kappa}\cdots f_1^{l-(a-2)\kappa} f_2^{l-(a-1)\kappa} \nonumber
\end{gather*}
are well-defined as elements of $U\hat{\mathfrak{n}}_-$. If $m=l-1-(a-1)\kappa$, then $F_{12}(l,a,\kappa)v\in V(m,k-m)$ is a singular vector of degree $(la,l(a-1))$ and if $m=-l-1+a\kappa$, then $F_{21}(l,a,\kappa)v\in V(m,k-m)$ is a singular vector of degree $(l(a-1),la)$.
\end{thm}

An explanation of the meaning of complex powers in these formulas see in~\cite{MFF}.

For example for $m=-2+\kappa$, we have
\begin{gather*}
F_{21}(1,1,\kappa)v=f_2v=\frac{e}{T}v,
\end{gather*}
and for $m=-2+2\kappa$, we have
\begin{gather*}
F_{21}(1,2,\kappa)v=f_2^{1+\kappa} f_1f_2^{1-\kappa}v=f\left(\frac{e}{T}\right)^2 v+(1+\kappa)\frac{h}{T} \frac{e}{T}v-(1+\kappa)\kappa\frac{e}{T^2}v.
\end{gather*}

\subsection{Shapovalov form}\label{contra}
The {\em Shapovalov form} on an $\widehat{{\mathfrak{sl}_2}}$ Verma module $V$ with generating vector $v$ is the unique symmetric bilinear form $S(\cdot,\cdot)$ on $V$ such that
\begin{gather*}
S(v,v)=1,\qquad S(f_ix,y)=S(x,e_iy) \qquad \text{for} \ i=1,2;\ x,y\in V.
\end{gather*}

For $\gamma\in{\mathbb Z}^2_{\geq 0}$, let $V^*_{\gamma}$ be the vector space dual to $V_{\gamma}$. Define $V^*=\oplus_\gamma V_{\gamma}^*$. The space $V^*$ is an $\widehat{{\mathfrak{sl}_2}}$-module with the $\widehat{{\mathfrak{sl}_2}}$-action defined by the formulas:
\begin{gather*}
\langle f_i\phi,x\rangle=\langle \phi,e_ix\rangle, \qquad \langle e_i\phi,x\rangle=\langle\phi,f_ix\rangle ,
\end{gather*}
where $\phi\in V^*$, $x\in V$, $i=1,2 $. The $\widehat{{\mathfrak{sl}_2}}$-module $V^*$ is called the {\em contragradient} Verma module.

The Shapovalov form $S$ considered as a map $S\colon V\longrightarrow V^*$ is a morphism of $\widehat{{\mathfrak{sl}_2}}$-modules.

\subsection[Bases in $V$ and $V^*$]{Bases in $\boldsymbol{V}$ and $\boldsymbol{V^*}$}\label{BASIS}
Let $V$ be an $\widehat{{\mathfrak{sl}_2}}$ Verma module $V$. For every $\gamma=(p_1,p_2)\in{\mathbb Z}^2_{\geq 0}$ with $p_1\ne p_2$, we fix a basis in the homogeneous component $V_{\gamma}\subset V$.

For $p_1>p_2$, we fix the basis
\begin{gather*}%\label{p_1>p_2}
\left\{\frac{f}{T^{i_1}}\cdots \frac{f}{T^{i_a}} \frac{h}{T^{j_1}}\cdots\frac{h}{T^{j_b}} \frac{e}{T^{k_1}}\cdots\frac{e}{T^{k_c}}v \right\},
\end{gather*}
where
\begin{gather}
0\leq i_a\leq i_{a-1}\leq\dots\leq i_1,\quad 1\leq j_b\leq j_{b-1}\leq\dots\leq j_1,\quad 1\leq k_c\leq k_{c-1}\leq\dots\leq k_1; \nonumber\\
\sum_{s=1}^a i_s + \sum_{s=1}^b j_s + \sum_{s=1}^c k_s + a - c = p_1, \qquad \sum_{s=1}^a i_s + \sum_{s=1}^b j_s + \sum_{s=1}^c k_s = p_2.\label{indices}
\end{gather}
For $p_1 < p_2$, we fix the basis
\begin{gather*}%\label{p_1<p_2}
\left\{\frac{e}{T^{k_1}}\cdots\frac{e}{T^{k_c}} \frac{h}{T^{j_1}}\cdots\frac{h}{T^{j_b}} \frac{f}{T^{i_1}}\cdots\frac{f}{T^{i_a}}v\right\},
\end{gather*}
with the indices satisfying~(\ref{indices}). Notice that for any $x\in\mathfrak{sl}_2$ the elements $\frac x{T^i}$ and $\frac x{T^j}$ commute.

These collections of vectors are bases by the Poincar\'{e}--Birkhoff--Witt theorem.

\looseness=-1 For any $\gamma$, we fix a basis in the $\gamma$-homogeneous component $V_\gamma^*\subset V^*$ as the basis dual of the basis in $V_\gamma$ specified above.
If $\{w_i\}$ is a basis in $V_\gamma$, then we denote by $\{(w_i)^*\}$ the dual basis in~$V^*_\gamma$.

\subsection{Main formula}
\begin{thm}[{\cite[Theorem 5.12]{SV2}}]\label{thm form}
For $m, k\in{\mathbb C}$ and $a\in{\mathbb Z}_{>0}$, the following identities hold in the contragradient Verma module $V(m,k-m)^*$,
\begin{gather}
\frac{f}{T^{a-1}} (v)^* = (m+(a-1)(k+2))\left(\frac{f}{T^{a-1}}v\right)^* \nonumber \\
\hphantom{\frac{f}{T^{a-1}} (v)^* =} + \sum_{\ell=1}^{a-1}\bigg[\frac{h}{T^\ell}\left(\frac{f}{T^{a-1-\ell}}v\right)^* + 2 \frac{e}{T^\ell}\mathop{\sum_{i+j=a-1-\ell}}_{i\geq j\geq 0} \left(\frac{f}{T^i}\frac{f}{T^j}v\right)^*\bigg], \label{id A}\\
\frac{e}{T^a} (v)^*=(a(k+2)-m-2) \left(\frac{e}{T^a}v\right)^* \nonumber\\
\hphantom{\frac{e}{T^a} (v)^*=} + \sum_{\ell=0}^{a-2}\bigg[{-} \frac{h}{T^{\ell+1}} \left(\frac{e}{T^{a-\ell-1}}v\right)^* + 2 \frac{f}{T^\ell}\mathop{\sum_{i+j=a-\ell}}_{i\geq j\geq 1} \left(\frac{e}{T^i}\frac{e}{T^j}v\right)^*\bigg],\label{id B}
\end{gather}
where $v$ is the generating vector of the Verma module $V(m,k-m)$.
\end{thm}

Theorem~\ref{thm form} was announced in \cite{SV2}. The proof of Theorem~\ref{thm form} is the main result of this paper. The theorem is proved in Section~\ref{sec PROOF}.

\begin{rem} The right-hand sides of formulas~(\ref{id A}) and~(\ref{id B}) have the factors $m+(a-1)(k+2)$ and $a(k+2)-m-2$. The vanishing of these factors corresponds to the resonance conditions $m_i+(a-1)\kappa=0$ and $m_{n+1}+2-a\kappa=0$ for the de Rham complex in Section~\ref{reson}, if we recall that $\kappa=k+2$.
\end{rem}

\begin{rem}Theorem \ref{thm form} says that the action of the element $\frac{f}{T^{a-1}}$ of degree $(a,a-1)$ on the covector $(v)^*$ can be expressed in terms of the actions of the elements $\frac{h}{T^{l}}$ and $\frac{e}{T^{l}}$ of smaller degree on some other covectors. Similarly the action of the element $\frac{e}{T^{a}}$ of degree $(a-1,a)$ on the covector~$(v)^*$ can be expressed in terms of the actions of the elements $\frac{h}{T^{l}}$, $\frac{f}{T^{l}}$ of smaller degree on some other~covectors.
\end{rem}

\subsection{Relation to Malikov--Feigin--Fuchs vectors}\label{rel to MFF}
Let
\begin{gather*}
S \colon \ V(m,k-m)\to V(m,k-m)^*
\end{gather*}
be the Shapovalov form.
Denote
\begin{gather*}
X_a(m,k-m) = S^{-1}\left((m+(a-1)(k+2))\left(\frac{f}{T^{a-1}}v\right)^*\right),\\
Y_a(m,k-m) = S^{-1}\left((m+2-a(k+2))\left(\frac{e}{T^a}v\right)^*\right).
\end{gather*}
For generic values of $m$ and $k$, the Shapovalov form $S$ is non-degenerate and $X_a$ and $Y_a$ are well defined elements of~$V(m,k-m)$. The chosen basis in $V(m,k-m)$ allows us to compare these vectors for different values of~$k$,~$m$. The vectors $X_a(m,k-m)$, $Y_a(m,k-m)$ are holomorphic functions of~$k$,~$m$ for generic~$k$,~$m$.

Recall the resonance lines in the $(m,k)$-plane, given by the equations
\begin{gather*}
m-l+1+(a-1)(k+2)=0, \qquad m+l+1-a(k+2)=0, \qquad k+2=0,
\end{gather*}
for some $a, l\in{\mathbb Z}_{>0}$, see Section \ref{reduc}.

\begin{thm}[{\cite[Theorem 6.2]{SV2}}]\label{limit}For $a\in{\mathbb Z}_{>0}$ let $(m_0, k_0)$ be a point of the line $ m+(a-1)(k+2)=0$, which does not belong to other resonance lines. Then the function $X_a(m,k-m)$ can be analytically continued to the point $(m_0,k_0)$, and $X_a(m_0,k_0-m_0)$ is a (nonzero) singular vector of $V(m_0,k_0-m_0)$, hence it is proportional to the Malikov--Feigin--Fuchs vector $F_{12}(1,a, k_0+2)$.

Similarly, for $a\in{\mathbb Z}_{>0}$ let $(m_0, k_0)$ be a point of the line $m+2-a(k+2)=0$, which does not belong to other resonance lines. Then the function $Y_a(m,k-m)$ can be analytically continued to the point $(m_0,k_0)$, and $Y_a(m_0,k_0-m_0)$ is a $($nonzero$)$ singular vector of $V(m_0,k_0-m_0)$, hence it is proportional to the Malikov--Feigin--Fuchs vector $F_{21}(1,a, k_0+2)$.
\end{thm}

\section{Homomorphism of complexes}\label{sec homs}

\subsection[Lie algebra $\mathfrak{sl}_2(U)$]{Lie algebra $\boldsymbol{\mathfrak{sl}_2(U)}$}\label{conf block}
Recall that $\{z_1,\dots,z_n,z_{n+1}=\infty\}$ are pairwise distinct points of the complex projective line~${\mathbb P}^1$ and $U={\mathbb P}^1 -\{z_1,\dots,z_n,z_{n+1}\}$. Fix local coordinates $t-z_1,\dots, t-z_n,1/t$ on ${\mathbb P}^1$ at these points, respectively.
Let $\mathfrak{sl}_2(U)$ be the Lie algebra of $\mathfrak{sl}_2$-valued rational functions on ${\mathbb P}^1$ regular on $U$, with the pointwise bracket. Thus, an element of $\mathfrak{sl}_2(U)$ has the form $e\otimes u_1 + h\otimes u_2+ f\otimes u_3$ with $u_i\in\Omega^0(U)$, and the bracket is defined by the formula $[x\otimes u_1, y\otimes u_2] = [x,y]\otimes (u_1u_2)$.

\subsection[$\mathfrak{sl}_2(U)$-modules]{$\boldsymbol{\mathfrak{sl}_2(U)}$-modules}
We say that an $\widehat{{\mathfrak{sl}_2}}$-module $W$ has the finiteness property, if for any $w\in W$ and $ x\in\mathfrak{sl}_2$, we have $xT^j \cdot w=0$ for all $j \gg 0$. For example, the contragradient Verma module has the finiteness property.

Let $W_1,\dots,W_{n+1}$ be $\widehat{{\mathfrak{sl}_2}}$-modules with the finiteness property. Then the Lie algebra~$\mathfrak{sl}_2(U)$ acts on $W_1\otimes\dots\otimes W_{n+1}$ by the formula
\begin{gather*}
x\otimes u \cdot (w_1\otimes\dots\otimes w_{n+1}) = \big([x\otimes u(t)]^{(z_1)} w_1\big)
\otimes w_2\otimes\dots\otimes w_{n+1}+ \cdots \nonumber\\
\hphantom{x\otimes u \cdot (w_1\otimes\dots\otimes w_{n+1}) =}{} + w_1\otimes\dots\otimes w_{n-1}\otimes \big([x\otimes u(t)]^{(z_n)} w_n\big)\otimes w_{n+1} \nonumber\\
\hphantom{x\otimes u \cdot (w_1\otimes\dots\otimes w_{n+1}) =}{} + w_1\otimes\dots\otimes w_n\otimes \big(\pi([x\otimes u(t)]^{(\infty)}) w_{n+1}\big), %\label{action}
\end{gather*}
where for $x\otimes u \in\mathfrak{sl}_2(U)$ the symbol $[x\otimes u(t)]^{(z_j)}$ denotes the Laurent expansion of $x\otimes u$ at $t=z_j$ and $[x\otimes u(t)]^{(\infty)}$ denotes the Laurent expansion at $t=\infty$; the symbol $\pi$ in the last term denotes the $\widehat{{\mathfrak{sl}_2}}$-automorphism defined in Section \ref{aut}.

The finiteness property of the tensor factors ensures that the actions of the Laurent series are well-defined.

The $\widehat{{\mathfrak{sl}_2}}$-action gives us a map
\begin{gather}\label{mult}
\mu\colon \ \mathfrak{sl}_2(U)\otimes \big(\otimes_{j=1}^{n+1} W_j\big) \to \otimes_{j=1}^{n+1} W_j .
\end{gather}

\subsection{Chain complex}
For a Lie algebra ${\mathfrak g}$ and a ${\mathfrak g}$-module $W$ we denote by $C_{\bullet}({\mathfrak g}, W)$ the standard chain complex of ${\mathfrak g}$ with coefficients in $W$, where
\begin{gather*}
 C_p({\mathfrak g}, W) = \wedge ^p{\mathfrak g} \otimes W, \\
{\rm d}(g_p\wedge \dots\wedge g_1\otimes w) = \sum_{i=1}^p(-1)^{i-1} g_p\wedge \dots\wedge \widehat{g_i}\wedge\dots \wedge g_1\otimes g_iw \\
\hphantom{{\rm d}(g_p\wedge \dots\wedge g_1\otimes w) =}{} + \sum_{1\leq i<j\leq p}(-1)^{i+j} g_p\wedge \dots\wedge \widehat{g_j}\wedge\dots\wedge \widehat{g_i}\wedge\dots \wedge g_1\otimes [g_j,g_i]w.
\end{gather*}

\subsection{Two complexes}

\subsubsection{} Let $m_1,\dots,m_n, k\in{\mathbb C} $, $k+2\ne 0$. Define $m_{n+1}=m_1+\dots+m_n-2$. For $j=1,\dots, n+1$, let $V_j$ be the $\widehat{{\mathfrak{sl}_2}}$ Verma module $V(m_j,k-m_j)$ and $V_j^*$ the corresponding contragradient Verma module. Consider the chain complex $C_{\bullet}\big(\mathfrak{sl}_2(U), \otimes_{j=1}^{n+1}V_j^*\big)$ and its last two terms{\samepage
\begin{gather*}%\label{l2t}
\to\ \mathfrak{sl}_2(U) \otimes \big({\otimes}_{j=1}^{n+1}V_j^*)\overset{{\rm d}}{\longrightarrow} \otimes_{j=1}^{n+1}V_j^* \to 0,
\end{gather*}
where ${\rm d}=\mu$, see formula (\ref{mult}).}

We assign degree $0$ to the term $\otimes_{j=1}^{n+1}V_j^*$ of this complex and assign degree $1$ to the differential~${\rm d}$, so that the whole complex sits in the non-positive area.

\subsubsection{} Consider the twisted de Rham complex in (\ref{dr com}) corresponding to $\kappa=k+2$ with degrees shifted by~1, namely, the complex $\Omega^{\bullet}(U)[1]$,
\begin{gather*}%\label{cder}
 0\to\Omega^0(U) \overset{\partial}{\longrightarrow} \Omega^1(U)\to 0,
\end{gather*}
where the shift $[1]$ means that we assign degree $p-1$ to the term $\Omega^p(U)$.

\subsection{Construction}\label{sec constr}
Define a linear map
\begin{gather*}
\eta^1\colon \ \Omega^1(U)\longrightarrow \otimes_{j=1}^{n+1} V_j^*
\end{gather*}
by the formulas
\begin{gather}
\frac{{\rm d}t}{(t-z_m)^{a}} \mapsto -\kappa (v_1)^*\otimes\dots\otimes \left(\frac{f}{T^{a-1}}v_m\right)^* \otimes\dots\otimes (v_{n+1})^*, \label{eta 1f} \\
t^{a-1}{\rm d}t \mapsto \kappa (v_1)^*\otimes\dots\otimes (v_n)^* \otimes \left(\frac{e}{T^{a}}v_{n+1}\right)^*, \label{eta 1i}
\end{gather}
for $a> 0$. Define a linear map
\begin{gather*}
\eta^0\colon \ \Omega^0(U)\longrightarrow \mathfrak{sl}_2(U)\otimes \big({\otimes}_{j=1}^{n+1} V_j^*\big)
\end{gather*}
by the formulas
\begin{gather}
\frac{1}{(t-z_m)^a} \mapsto \frac{f}{(t-z_m)^a}\otimes (v_1)^*\otimes\dots\otimes (v_{n+1})^*\nonumber\\
\hphantom{\frac{1}{(t-z_m)^a} \mapsto}{} -\sum_{l=1}^a\bigg[\frac{e}{(t-z_m)^l}\otimes (v_1)^*\otimes \dots\otimes 2\sum_{i+j=a-l \atop i\geq j\geq 0} \left(\frac{f}{T^i}\frac{f}{T^j}v_m\right)^*\otimes\dots\otimes (v_{n+1})^* \nonumber \\
\hphantom{\frac{1}{(t-z_m)^a} \mapsto}{} {} +\frac{h}{(t-z_m)^l}\otimes (v_1)^*\otimes \dots\otimes \left(\frac{f}{T^{a-l}}v_m\right)^*\otimes \dots\otimes (v_{n+1})^* \bigg],\label{eta 0f}
\end{gather}
for $a>0$;
\begin{gather}
t^a \mapsto ft^a\otimes (v_1)^*\otimes\dots\otimes (v_{n+1})^*\nonumber \\
\hphantom{t^a \mapsto}{} - \sum_{l=0}^{a-2}\bigg[et^l\otimes (v_1)^*\otimes\dots\otimes
(v_n)^*\otimes 2\sum_{i+j=a-l,\atop i\geq j\geq 1} \left(\frac{e}{T^i} \frac{e}{T^j}v_{n+1}\right)^* \nonumber \\
\hphantom{t^a \mapsto}{} + ht^{l+1}\otimes (v_1)^*\otimes\dots\otimes (v_n)^*\otimes \left(\frac{e}{T^{a-l-1}}v_{n+1}\right)^* \bigg],\label{eta 0i}
\end{gather}
for $a\geq 0$.

\begin{thm}[{\cite[Theorem 5.12]{SV2}}]\label{thm hom} Formulas \eqref{eta 1f}--\eqref{eta 0i} define a homomorphism of complexes $\eta\colon \Omega^{\bullet}(U)[1] \to C_{\bullet}\big(\mathfrak{sl}_2(U);\otimes_{j=1}^{n+1} V_j^*\big)$, namely we have
\begin{gather*}%\label{homom}
{\rm d} \eta^0 = \eta^1\partial.
\end{gather*}
The homomorphism is injective.
\end{thm}

Theorem \ref{thm hom} was announced in \cite{SV2}. Here is a proof of the theorem.

\begin{proof} First we calculate $ \eta^1(\partial((t-z_p)^{-a}))$,
\begin{gather*}
\frac 1{(t-z_p)^{a}} \overset{\partial} \mapsto -\frac{1}{\kappa}(m_p+a\kappa) \frac{{\rm d}t}{(t-z_p)^{a+1}} +\frac{1}{\kappa} \sum_{k=1}^a \sum_{j\neq p}\frac{m_j}{(z_j-z_p)^k}
\frac{{\rm d}t}{(t-z_p)^{a+1-k}} \nonumber \\
\hphantom{\frac 1{(t-z_p)^{a}} \overset{\partial} \mapsto}{} - \frac{1}{\kappa} \sum_{j\neq p} \frac{m_j}{(z_j-z_p)^a}\frac{{\rm d}t}{t-z_j} \nonumber \\
\hphantom{\frac 1{(t-z_p)^{a}}}{}\overset{\eta^1} \mapsto (m_p+\kappa a) (v_1)^* \otimes\dots\otimes\left(\frac{f}{T^a}v_p \right)^* \otimes\dots\otimes (v_{n+1})^* \nonumber \\
\hphantom{\frac 1{(t-z_p)^{a}} \overset{\partial} \mapsto}{} - \sum_{k=1}^a \sum_{j \neq p}\frac{ m_j}{(z_j-z_p)^k} (v_1)^* \otimes\dots\otimes \left(\frac{f}{T^{a-k}} v_p \right)^* \otimes\dots\otimes (v_{n+1})^* \nonumber\\
\hphantom{\frac 1{(t-z_p)^{a}} \overset{\partial} \mapsto}{} + \sum_{j\neq p}\frac{ m_j}{(z_j-z_p)^a} (v_1)^* \otimes\dots\otimes (fv_j)^* \otimes\dots\otimes (v_{n+1})^*. \nonumber
\end{gather*}
Then we calculate ${\rm d}\big(\eta^0((t-z_p)^{-a})\big)$,
\begin{gather*}
\frac1{(t-z_p)^{a}} \overset{\eta^0}\mapsto \frac{f}{(t-z_p)^a} \otimes (v_1)^* \otimes\dots\otimes (v_{n+1})^* \nonumber \\
\hphantom{\frac1{(t-z_p)^{a}} \overset{\eta^0}\mapsto}{} - \sum_{l=1}^a \bigg[ \frac{h}{(t-z_p)^l} \otimes (v_1)^* \otimes\dots\otimes \left( \frac{f}{T^{a-l}} v_p \right)^* \otimes\dots\otimes (v_{n+1})^*
\nonumber \\
\hphantom{\frac1{(t-z_p)^{a}} \overset{\eta^0}\mapsto}{} + \frac{e}{(t-z_p)^l} \otimes 2 \sum_{i+j = a-l \atop i \geq j \geq 0} (v_1)^* \otimes\dots\otimes \left(\frac{f}{T^i} \frac{f}{T^j} v_p \right)^* \otimes\dots\otimes (v_{n+1})^* \bigg] \nonumber \\
\hphantom{\frac1{(t-z_p)^{a}}}{}
\overset{\mu} \mapsto (v_1)^* \otimes\dots\otimes \bigg[ (m_p + a(k+2)) \left(\frac{f}{T^a} v_p \right)^* \nonumber \\
\hphantom{\frac1{(t-z_p)^{a}} \overset{\eta^0}\mapsto}{} + \sum_{l=1}^a \bigg[\frac{h}{T^l}\left(\frac{f}{T^{a-l}}v_p\right)^*+
2\frac{e}{T^l}\sum_{i+j=a-l \atop i\geq j\geq 0 } \left(\frac{f}{T^i}\frac{f}{T^j} v_p\right)^*\bigg] \bigg] \otimes\dots\otimes (v_{n+1})^* \nonumber \\
\hphantom{\frac1{(t-z_p)^{a}} \overset{\eta^0}\mapsto}{} + \sum_{j\neq p} \frac{m_j}{(z_j-z_p)^a} (v_1)^* \otimes\dots\otimes (fv_j)^* \otimes\dots\otimes (v_{n+1})^* \nonumber \\
\hphantom{\frac1{(t-z_p)^{a}} \overset{\eta^0}\mapsto}{}
 - \sum_{l=1}^a \bigg[ (v_1)^* \otimes\dots\otimes \frac{h}{T^l}\left(\frac{f}{T^{a-l}}v_p \right)^* \otimes\dots\otimes (v_{n+1})^* \nonumber \\
\hphantom{\frac1{(t-z_p)^{a}} \overset{\eta^0}\mapsto}{} + (v_1)^* \otimes\dots\otimes \sum_{i+j=a-l \atop i\geq j\geq 0} 2\frac{e}{T^l} \left(\frac{f}{T^i}\frac{f}{T^j}v_p\right)^*\otimes\dots\otimes (v_{n+1})^*
\nonumber \\
\hphantom{\frac1{(t-z_p)^{a}} \overset{\eta^0}\mapsto}{} + \sum_{j \neq p} \frac{m_j}{(z_j-z_p)^l} (v_1)^* \otimes\dots\otimes \left(\frac{f}{T^{a-l}} v_p \right)^* \otimes\dots\otimes (v_{n+1})^* \bigg] \nonumber \\
\hphantom{\frac1{(t-z_p)^{a}}}{} = (\kappa a + m_p) (v_1)^* \otimes\dots\otimes\left(\frac{f}{T^a}v_p \right)^* \otimes\dots\otimes (v_{n+1})^* \nonumber \\
\hphantom{\frac1{(t-z_p)^{a}} \overset{\eta^0}\mapsto}{} - \sum_{l=1}^a \sum_{j \neq p} \frac{m_j}{(z_j-z_p)^l} (v_1)^* \otimes\dots\otimes \left(\frac{f}{T^{a-l}} v_p \right)^* \otimes\dots\otimes
(v_{n+1)}^* \nonumber \\
\hphantom{\frac1{(t-z_p)^{a}} \overset{\eta^0}\mapsto}{} + \sum_{j\neq p} \frac{ m_j}{(z_j-z_p)^a} (v_1)^* \otimes\dots\otimes (fv_j)^* \otimes\dots\otimes (v_{n+1})^*.
\end{gather*}
In this calculation we use formula (\ref{id A}) to express the action of $\frac{f}{T^a}$ on $(v_{p})^*$. These formulas show that ${\rm d}\big(\eta^0((t-z_p)^{-a})\big) = \eta^1\big(\partial((t-z_p)^{-a})\big)$.

Now we calculate $\eta^1(\partial(t^a))$,
\begin{gather*}
t^a \overset{\partial}{\mapsto} \frac{1}{\kappa}\bigg(a\kappa - \sum_{j=1}^n m_j\bigg) t^{a-1} {\rm d}t - \frac{1}{\kappa} \sum_{s=1}^{a-1} \sum_{j=1}^n m_j z_j^s
t^{a-s-1} {\rm d}t - \frac{1}{\kappa} \sum_{j=1}^n m_j z_j^a \frac{{\rm d}t}{t-z_j} \nonumber \\
\hphantom{t^a}{} \overset{\eta^1}{\mapsto} \bigg(a\kappa - \sum_{j=1}^n m_j\bigg) (v_1)^* \otimes \dots \otimes (v_n)^* \otimes \left(\frac{e}{T^a} v_{n+1}\right)^*\nonumber \\
\hphantom{t^a \overset{\partial}{\mapsto}}{}- \sum_{s=1}^{a-1} \sum_{j=1}^n m_j z_j^s (v_1)^* \otimes \dots \otimes (v_n)^* \otimes \left(\frac{e}{T^{a-s}} v_{n+1}\right)^* \nonumber \\
\hphantom{t^a \overset{\partial}{\mapsto}}{} + \sum_{j=1}^n m_j z_j^a (v_1)^* \otimes \dots \otimes (f v_j)^* \otimes\dots\otimes (v_{n+1})^*.
\end{gather*}

Then we calculate ${\rm d}\big(\eta^0(t^a)\big)$,
\begin{gather*}
t^a \overset{\eta^0}\mapsto ft^a\otimes (v_1)^*\otimes\dots\otimes (v_{n+1})^* - \sum_{l=0}^{a-2}\bigg[ht^{l+1}\otimes (v_1)^*\otimes\dots\otimes (v_n)^*\otimes \left(\frac{e}{T^{a-l-1}}v_{n+1}\right)^* \nonumber \\
\hphantom{t^a \overset{\eta^0}\mapsto}{} + et^l\otimes (v_1)^*\otimes\dots\otimes (v_n)^*\otimes 2\sum_{i+j=a-l\atop i\geq j\geq 1} \left(\frac{e}{T^i} \frac{e}{T^j}v_{n+1}\right)^* \bigg]
\nonumber \\
\hphantom{t^a}{}
\overset{\mu} \mapsto 	 \sum_{j=1}^n m_j z_j^a (v_1)^* \otimes \dots \otimes (f v_j)^* \otimes\dots\otimes (v_{n+1})^* \nonumber \\
\hphantom{t^a \overset{\eta^0}\mapsto}{} + (v_1)^* \otimes \dots \otimes (v_n)^* \otimes \bigg[\left(-m_{n+1}-2+a(k+2)\right) \left(\frac{e}{T^a}v_{n+1} \right)^* \nonumber \\
\hphantom{t^a \overset{\eta^0}\mapsto}{}+ \sum_{l=0}^{a-2} \bigg[{-} \frac{h}{T^{l+1}} \left(\frac{e}{T^{a-l-1}} v_{n+1} \right)^* + 2\frac{f}{T^l} \sum_{i+j=a-l\atop i\geq j\geq1} \left(\frac{e}{T^i} \frac{e}{T^j} v_{n+1} \right)^* \bigg] \bigg] \nonumber \\
\hphantom{t^a \overset{\eta^0}\mapsto}{}
-(v_1)^*\otimes\dots\otimes (v_n)^* \otimes \sum_{l=0}^{a-2} \bigg[ 2\frac{f}{T^l} \sum_{i+j=a-l\atop i\geq j\geq1} \left(\frac{e}{T^i} \frac{e}{T^j} v_{n+1} \right)^* - \frac{h}{T^{l+1}} \left(\frac{e}{T^{a-l-1}} v_{n+1} \right)^* \bigg] \nonumber \\
\hphantom{t^a \overset{\eta^0}\mapsto}{} - \sum_{s=1}^{a-1} \sum_{j=1}^n m_j z_j^s (v_1)^* \otimes \dots \otimes (v_n)^* \otimes \left(\frac{e}{T^{a-s}} v_{n+1}\right)^* \nonumber \\
\hphantom{t^a}{} = \bigg(a\kappa - \sum_{j=1}^n m_j\bigg) (v_1)^* \otimes \dots \otimes (v_n)^* \otimes \left(\frac{e}{T^a} v_{n+1}\right)^* \nonumber \\
\hphantom{t^a \overset{\eta^0}\mapsto}{} - \sum_{s=1}^{a-1} \sum_{j=1}^n m_j z_j^s (v_1)^* \otimes \dots \otimes (v_n)^* \otimes \left(\frac{e}{T^{a-s}} v_{n+1}\right)^* \nonumber \\
\hphantom{t^a \overset{\eta^0}\mapsto}{} + \sum_{j=1}^n m_j z_j^a (v_1)^* \otimes \dots \otimes (f v_j)^* \otimes\dots\otimes (v_{n+1})^*. \nonumber
\end{gather*}
In this calculation we use formula~(\ref{id B}) to express the action of $\frac{e}{T^a}$ on $(v_{n+1})^*$. Notice also that calculating the action on $V_{n+1}^*$ we use the automorphism~$\pi$, see Section~\ref{aut}.These formulas show that ${\rm d}\big(\eta^0(t^a)\big) = \eta^1(\partial(t^a))$.

Clearly the maps $\eta^1$, $\eta^2$ are injective. Theorem \ref{thm hom} is proved.
\end{proof}

\subsection{Image of logarithmic subcomplex}
Under the monomorphism $\eta$ of Theorem \ref{thm hom} the image of the logarithmic subcomplex $(\Omega^\bullet_{\log}(U), \partial)$ is the chain complex $C_{\bullet}\big({\mathfrak n}_-, \otimes_{j=1}^{n+1}V_j^*\big)$ of the nilpotent subalgebra ${\mathfrak n}_-\subset\mathfrak{sl}_2$ generated by~$f$. More precisely, we have
\begin{gather*}
\eta \colon \ 1\mapsto f\otimes (v_1)^*\otimes\dots\otimes(v_{n+1})^*, \qquad \frac{{\rm d}t}{x-t_j}\mapsto -\kappa(v_1)^*\otimes\dots\otimes (fv_j)^*\otimes\dots\otimes(v_{n+1})^*,
\end{gather*}
$j=1,\dots,n$, and
\begin{gather*}
\mu\colon \ f\otimes (v_1)^*\otimes\dots\otimes(v_{n+1})^* \mapsto \sum_{j=1}^n m_j (v_1)^*\otimes\dots\otimes(fv_j)^*\otimes\dots\otimes(v_{n+1})^*.
\end{gather*}

Far-reaching generalizations of this identification of the logarithmic subcomplex with the chain complex of the nilpotent Lie algebra ${\mathfrak n}_-$ see in~\cite{SV1}.

\section{Proof of Theorem \ref{thm form}}\label{sec PROOF}

\subsection{Formula (\ref{id B}) follows from formula (\ref{id A})}
The Lie algebra $\widehat{{\mathfrak{sl}_2}}$ has an automorphism $\rho$, corresponding to the involution of the Dynkin diagram:
\begin{gather*}
\rho(e_{i})=e_{3-i}, \qquad \rho(f_{i})=f_{3-i}, \qquad \rho(h_{i})=h_{3-i}, \qquad i=1,2.
\end{gather*}
We have $\rho^2={\rm id}$. In other words, $\rho$ acts by the formulas
\begin{gather*}
e\leftrightarrow fT, \qquad f \leftrightarrow e T^{-1}, \qquad h \leftrightarrow c-h.
\end{gather*}

\begin{lem}\label{lem elts}For $i \in \mathbb{Z}_{> 0}$, we have
\begin{gather*}
\rho \colon \ \frac{f}{T^i} \mapsto \frac{e}{T^{i+1}}, \qquad \frac{e}{T^i} \mapsto \frac{f}{T^{i-1}}, \qquad \frac{h}{T^i} \mapsto -\frac{h}{T^i}.
\end{gather*}
\end{lem}
\begin{proof}
We have
\begin{gather*}
\frac{f}{T} = \frac12 \left[f, \frac hT \right] = \frac 12 \left[ f, \left[\frac eT, f \right] \right] \overset{\delta}{\longrightarrow} \frac 12 \left[ \frac eT, \left[f, \frac eT \right] \right] = \frac 12 \left[ \frac eT, - \frac hT \right] = \frac{e}{T^2}, \nonumber \\
\frac{f}{T^i} = \frac 12 \left[ \frac{f}{T^{i-1}}, \left[\frac{e}{T}, f\right] \right] \overset{\delta}{\longrightarrow} \frac 12 \left[ \frac{e}{T^i}, \left[f, \frac eT \right] \right] = \frac 12 \left[ \frac{e}{T^i}, - \frac hT \right] = \frac{e}{T^{i+1}}. \nonumber
\end{gather*}
Similarly we prove that $\rho\big(\frac{e}{T^i}\big) = \frac{f}{T^{i-1}}$, $\rho \big(\frac{h}{T^i}\big) = -\frac{h}{T^i}$.
\end{proof}

For $m\in{\mathbb C}$, let $\sigma_m \colon \widehat{{\mathfrak{sl}_2}} \to \operatorname{End}( V(m,k-m))$ be the Verma module structure. Let $\sigma_m\circ \rho \colon \widehat{{\mathfrak{sl}_2}} \to \operatorname{End}( V(m,k-m))$ be the twisted module structure.

Clearly the $\widehat{{\mathfrak{sl}_2}}$-modules $(\sigma_m\circ \rho, V(m,k-m))$ and $(\sigma_{m-k}, V(m-k,m))$ are isomorphic. If $v_m\in V(m,k-m)$ and $v_{k-m}\in V(k-m,m)$ are generating vectors, then an isomorphism $\chi\colon (\sigma_m\circ \rho, V(m,k-m)) \to (\sigma_{m-k}, V(m-k,m))$ is defined by the formula,
\begin{gather*}
f_{i_l}\cdots f_{i_1}v_{k-m} \mapsto f_{3-i_l}\cdots f_{3-i_1}v_{m},
\end{gather*}
for any $i_1,\dots,i_l\in\{1,2\}$. The isomorphism $\chi$ restricts to isomorphisms of the graded components, $V(k-m,m)_{(p_1,p_2)} \to V(m,k-m)_{(p_2,p_1)}$.

In Section \ref{BASIS} we fixed bases of the homogeneous components $V_{(p_1,p_2)}$ with $p_1\ne p_2$ of any Verma module $V$. By Lemma~\ref{lem elts}, under the isomorphism $\chi$ the chosen basis of $V(k-m, m)_{(p_1,p_2)}$ is mapped to the chosen basis of $V(m,k-m)_{(p_2,p_1)}$ up to multiplication of the basis vectors by~$\pm 1$. This~$\pm 1$ appears due to the formula $\rho\big(\frac h{T^i}\big) = - \frac h{T^i}$. In particular, we have
\begin{gather*}
\chi \colon \ \frac f{T^i}v_{k-m} \mapsto \frac e{T^{i+1}}v_{m}, \qquad
\frac f{T^i} \frac f{T^j} v_{k-m} \mapsto \frac e{T^{i+1}}\frac e{T^{j+1}}v_{m}.
\end{gather*}

Let $ \sigma^*_m \colon \widehat{{\mathfrak{sl}_2}} \to \operatorname{End}( V(m,k-m)^*)$ be the contragradient Verma module structure. Let $ \sigma^*_m\circ \rho \colon \widehat{{\mathfrak{sl}_2}} \to \operatorname{End}( V(m,k-m)^*)$ be the twisted module structure. The isomorphism $\chi$ induces an isomorphism of modules $\chi^* \colon (\sigma^*_m\circ \rho, V(m,k-m)^*) \to (\sigma^*_{m-k}, V(m-k,m)^*)$.

In Section~\ref{BASIS} we fixed bases in the homogeneous components $V^*_{(p_1,p_2)}$ with $p_1\ne p_2$ of any contragradient Verma module~$V^*$. Under the isomorphism $\chi^*$, the chosen basis of $V(k-m, m)^*_{(p_1,p_2)}$ is mapped to the chosen basis of $V(m,k-m)^*_{(p_2,p_1)}$ up to multiplication of the basis vectors by~$\pm 1$. In particular, we have
\begin{gather*}
 \chi^* \colon \ \left(\frac f{T^i}v_{k-m}\right)^* \mapsto \left(\frac e{T^{i+1}}v_{m}\right)^*, \qquad
\left(\frac f{T^i} \frac f{T^j} v_{k-m}\right)^* \mapsto \left(\frac e{T^{i+1}}\frac e{T^{j+1}}v_{m}\right)^*.
\end{gather*}

Assume that the relation in formula (\ref{id A}) holds in every contragradient Verma module $V^*$. Then in $V(k-m, m)^*$ it takes the form
\begin{gather*}
\frac{f}{T^{a-1}} (v_{k-m})^* = (-m-2+ a(k+2))\left(\frac{f}{T^{a-1}}v_{k-m}\right)^* \nonumber \\
\hphantom{\frac{f}{T^{a-1}} (v_{k-m})^* =}{} + \sum_{\ell=1}^{a-1} \bigg[\frac{h}{T^\ell}\left(\frac{f}{T^{a-1-\ell}}v_{k-m}\right)^* + 2 \frac{e}{T^\ell}\mathop{\sum_{i+j=a-1-\ell}}_{i\geq j\geq 0}
\left(\frac{f}{T^i}\frac{f}{T^j}v_{k-m}\right)^*\bigg].%\label{id Am}
\end{gather*}
The isomorphism $\chi^*$ sends this relation to the relation in $V(m,k-m)^*$,
\begin{gather*}
\frac{e}{T^{a}} (v_{m})^* = (-m-2+ a(k+2))\left(\frac{e}{T^{a}}v_{m}\right)^* \nonumber\\
\hphantom{\frac{e}{T^{a}} (v_{m})^* =}{} + \sum_{\ell=1}^{a-1} \bigg[{-}\frac{h}{T^\ell}\left(\frac{e}{T^{a-\ell}}v_{m}\right)^* + 2 \frac{f}{T^{\ell-1}}\mathop{\sum_{i+j=a-1-\ell}}_{i\geq j\geq 0}
\left(\frac{e}{T^{i+1}}\frac{e}{T^{j+1}}v_{m}\right)^*\bigg],%\label{id Amm}
\end{gather*}
which is exactly the relation in formula~(\ref{id B}). Thus formula (\ref{id A}) implies formula~(\ref{id B}).

\subsection{Auxiliary lemma}\label{subsec aux lemma}
Let
\begin{gather*}
V=V(m,k-m) \qquad \text{and} \qquad V^* = V(m,k-m)^*.
\end{gather*}

\begin{lem} \label{lem pere} For $x\in V$, $\phi\in V^*$, $k\in{\mathbb Z}_{\geq 0}$, we have
\begin{gather*}
\left\langle \frac{f}{T^k} \varphi, x \right\rangle = \big\langle \varphi, eT^k x \big\rangle,
 \qquad \left\langle \frac{e}{T^k} \varphi, x \right\rangle = \big\langle \varphi, fT^k x \big\rangle, \qquad
\left\langle \frac{h}{T^k} \varphi, x \right\rangle = \big\langle \varphi, hT^k x \big\rangle. %\label{perek}
\end{gather*}
\end{lem}
\begin{proof}The proof is by induction. We prove the first equality, the others are proved similarly.

We have $[f_2, f_1] = \frac{h}{T}$, hence $[f_1, [f_2, f_1]] = \frac{2f}{T}$. Similarly
 $[e_1, [e_2, e_1]] = 2e T$. So for $k=1$, we have
\begin{gather*}
\left\langle \frac{f}{T} \varphi, x \right\rangle = \left\langle \frac{1}{2} [f_1, [f_2, f_1]] \varphi, x \right\rangle = \left\langle \varphi, \frac{1}{2} [[e_1, e_2], e_1] x \right\rangle = \langle \varphi, eT x \rangle.
\end{gather*}
We have $\big[f_2, \frac{f}{T^{k-1}}\big] = \frac{h}{T^k}$, hence $\big[f_1, \big[f_2, \frac{f}{T^{k-1}} \big]\big] = \frac{2f}{T^k}$. Similarly, $\big[\big[eT^{k-1},fT\big], e\big] = 2e T^k$. Then
\begin{gather*}
\left\langle \frac{f}{T^k} \varphi, x \right\rangle = \left\langle \frac 12 \left[f_1, \left[f_2, \frac{f}{T^{k-1}}\right]\right] \varphi, x \right\rangle = \big\langle \varphi, \big[\big[eT^{k-1},e_2\big], e_1\big] x \big\rangle = \big\langle \varphi, eT^k x \big\rangle.\tag*{\qed}
\end{gather*}\renewcommand{\qed}{}
\end{proof}

\subsection{The structure of the proof of formula (\ref{id A})}\label{subsec structure} We reformulate formula (\ref{id A}) as
\begin{gather}
(m+(a-1)(k+2))\left(\frac{f}{T^{a-1}}v\right)^* \nonumber\\
\qquad{} =\frac{f}{T^{a-1}} (v)^* - \sum_{\ell=1}^{a-1}\bigg[\frac{h}{T^\ell}\left(\frac{f}{T^{a-1-\ell}}v\right)^* + 2 \frac{e}{T^\ell}\mathop{\sum_{i+j=a-1-\ell}}_{i\geq j\geq 0} \left(\frac{f}{T^i}\frac{f}{T^j}v\right)^*\bigg], \label{AA}
\end{gather}
and will prove it in this form.

Each term in (\ref{AA}) is an element of the homogeneous component $V^*_{(a,a-1)}$. In Section~\ref{BASIS} we specified a basis of the dual component $V_{(a,a-1)}$. We will calculate the value of the right-hand side in~(\ref{AA}) on an arbitrary basis vector and will obtain the value of the left-hand side on that vector.

The basis in $V_{(a,a-1)}$ consists of the vectors
\begin{gather*}
\frac{f}{T^{i_1}}\cdots\frac{f}{T^{i_r}} \frac{h}{T^{j_1}}\cdots\frac{h}{T^{j_s}} \frac{e}{T^{l_1}}\cdots\frac{e}{T^{l_{r-1}}}v, %\label{bb}
\end{gather*}
where
\begin{gather*}
0\leq i_r\leq i_{r-1}\leq\dots\leq i_1, \qquad 1\leq j_s\leq j_{s-1}\leq\dots\leq j_1,\qquad 1\leq l_{r-1}\leq l_{r-2}\leq\dots\leq l_1;\nonumber\\
 \sum_{u=1}^r i_u + \sum_{u=1}^s j_u + \sum_{u=1}^{r-1} l_u = a-1. %\label{indice}
\end{gather*}
We partition the basis in four groups. Group~${\rm O}$ consists of the single basis vector $\frac{f}{T^{a-1}}v$. Group~${\rm I}$ consists of all basis vectors with $r=1$, but different from $\frac{f}{T^{a-1}}v$. Group ${\rm II}$ consists of all basis vectors with $r=2$. Group ${\rm III}$ consists of all basis vectors with $r\geq 3$.

Notice that the value of the left-hand side of~(\ref{AA}) on the basis vector~$\frac{f}{T^{a-1}}v$ equals $m + (a-1)(k+2)$. Hence we need to show that the value of the right-hand side on the basis vector $\frac{f}{T^{a-1}}v$ equals $m + (a-1)(k+2)$. Similarly the value of the left-hand side on any basis vector of Groups~\mbox{I--III} equals zero. Hence we need
to prove that the value of the right-hand side on any basis vector of Groups~I--III equals zero. These four statements are the content of Propositions~\ref{Pr4},~\ref{Pr3},~\ref{Pr2}, and~\ref{Pr1} below. These propositions prove Theorem~\ref{thm form}.

\subsection[Group ${\rm O}$]{Group $\boldsymbol{{\rm O}}$}\label{subsec group O}
\begin{prop}\label{Pr4}The value of the right-hand side of \eqref{AA} on the basis vector $\frac{f}{T^{a-1}}v$ equals $m + (a-1)(k+2)$.
\end{prop}
\begin{proof}By Lemma \ref{lem pere} we have{\samepage
\begin{gather*}
\left\langle \frac{f}{T^{a-1}} (v)^*, \frac{f}{T^{a-1}} v \right\rangle = \left\langle (v)^*, eT^{a-1} \frac{f}{T^{a-1}} v \right\rangle \\
\hphantom{\left\langle \frac{f}{T^{a-1}} (v)^*, \frac{f}{T^{a-1}} v \right\rangle} = \left\langle (v)^*, \left[ h+(a-1)c + \frac{f}{T^{a-1}} eT^{a-1} \right] v \right\rangle = m + (a-1)k,
\end{gather*}
since $eT^{a-1} v$ is of degree $(-a, -a+1)$, hence zero.}

By Lemma \ref{lem pere}, for $ \ell \in \{ 1, \dots, a-1 \}$ we have
\begin{gather*}
 \bigg\langle \frac{h}{T^{\ell}} \left( \frac{f}{T^{a-1-\ell}} v \right)^*, \frac{f}{T^{a-1}} v \bigg\rangle
 = \bigg\langle \left( \frac{f}{T^{a-1-\ell}} v \right)^*, hT^{\ell} \frac{f}{T^{a-1}} v \bigg\rangle \\
 \hphantom{\bigg\langle \frac{h}{T^{\ell}} \left( \frac{f}{T^{a-1-\ell}} v \right)^*, \frac{f}{T^{a-1}} v \bigg\rangle}{} = \bigg\langle \left( \frac{f}{T^{a-1-\ell}} v \right)^*, -2 \frac{f}{T^{a-1-\ell}} v \bigg\rangle = -2.
\end{gather*}
By Lemma \ref{lem pere} for $ \ell \in \{ 1, \dots, a-1 \}$ we have
\begin{gather*}
\bigg\langle \frac{e}{T^{\ell}} \mathop{\sum_{i+j=a-1-\ell}}_{i\geq j\geq 0} \left(\frac{f}{T^i} \frac{f}{T^j}v \right)^*, \frac{f}{T^{a-1}} v \bigg\rangle = \bigg\langle \mathop{\sum_{i+j=a-1-\ell}}_{i\geq j\geq 0} \left(\frac{f}{T^i} \frac{f}{T^j}v \right)^*, fT^{\ell} \frac{f}{T^{a-1}} v \bigg\rangle \\
\hphantom{\bigg\langle \frac{e}{T^{\ell}} \mathop{\sum_{i+j=a-1-\ell}}_{i\geq j\geq 0} \left(\frac{f}{T^i} \frac{f}{T^j}v \right)^*, \frac{f}{T^{a-1}} v \bigg\rangle} = \bigg\langle \mathop{\sum_{i+j=a-1-\ell}}_{i\geq j\geq 0} \left(\frac{f}{T^i} \frac{f}{T^j}v \right)^*, \frac{f}{T^{a-1}} fT^{\ell} v \bigg\rangle =0,
\end{gather*}
since $fT^{\ell} v$ is of degree $ (-\ell+1, -\ell) \leq (0,-1)$, hence zero. Therefore,
\begin{gather*}
\bigg\langle \frac{f}{T^{a-1}} (v)^* - \sum_{\ell=1}^{a-1} \bigg[\frac{h}{T^\ell}\left(\frac{f}{T^{a-1-\ell}}v\right)^* + 2\frac{e}{T^\ell} \mathop{\sum_{i+j=a-1-\ell}}_{i\geq j\geq 0} \left(\frac{f}{T^i}\frac{f}{T^j}v\right)^* \bigg], \frac{f}{T^{a-1}} v \bigg\rangle \\
\qquad{}= m + (a-1)k - (-2)(a-1) = m + (a-1)(k+2).
\end{gather*}
Proposition \ref{Pr4} is proved.
\end{proof}

\subsection{Group I}
\begin{prop}\label{Pr3}The value of the right-hand side of \eqref{AA} on any basis vector of Group ${\rm I}$ equals zero.
\end{prop}
\begin{proof}Group ${\rm I}$ consists of basis vectors of the form
\begin{gather*}
w = \frac{f}{T^{a-1-n}}\frac{h}{T^{j_1}} \cdots \frac{h}{T^{j_s}} v, \qquad \text{where} \quad
n \in \{1, \dots, a-1 \},\quad j_1 + \dots + j_s = n, \quad j_i \geq 1.
\end{gather*}

\begin{lem}\label{Group 3, s=1} In the notation above, if $s =1$, then
\begin{gather}
\bigg\langle \frac{f}{T^{a-1}} (v)^*, w \bigg\rangle = 2 n k, \nonumber\\ %\label{G1s11}\\
\bigg\langle \frac{h}{T^{\ell}} \left(\frac{f}{T^{a-1-\ell}}v\right)^*, w \bigg\rangle = \begin{cases}
 2 n k, & \text{if $\ell=n$,}\\
 - 4, & \text{if $\ell > a-1-n$,} \\
 0, & \text{if $\ell \leq a-1-n$,}
 \end{cases} \label{G1s12}
\\
\bigg\langle \frac{e}{T^{\ell}} \mathop{\sum_{i+j=a-1-\ell}}_{i\geq j\geq 0} \left(\frac{f}{T^i}\frac{f}{T^j}v\right)^*, w \bigg\rangle =
 \begin{cases}
 2, & \text{if $\ell \leq n$,}\\
 0, & \text{if $\ell > n$. }
 \end{cases} \nonumber %\label{G1s13}
\end{gather}
\end{lem}

Note that the first line in (\ref{G1s12}) is not mutually exclusive with the second and third lines in~(\ref{G1s12}).

\begin{proof}
We have $w = \frac{f}{T^{a-1-n}}\frac{h}{T^n} v $. Then
\begin{gather*}
\bigg\langle \frac{f}{T^{a-1}} (v)^*, \frac{f}{T^{a-1-n}}\frac{h}{T^n} v \bigg\rangle = \bigg\langle (v)^*, eT^{a-1} \frac{f}{T^{a-1-n}}\frac{h}{T^n} v \bigg\rangle \\
\hphantom{\bigg\langle \frac{f}{T^{a-1}} (v)^*, \frac{f}{T^{a-1-n}}\frac{h}{T^n} v \bigg\rangle }
= \bigg\langle (v)^*, \left[ hT^n + \frac{f}{T^{a-1-n}} eT^{a-1} \right] \frac{h}{T^n} v \bigg\rangle.
\end{gather*}
Note that $eT^{a-1} \frac{h}{T^n}v$ is of degree $(-a+n, -a+1+n) \leq (-1,0)$, so that $eT^{a-1} \frac{h}{T^n} v = 0$. Hence
\begin{gather*}
\bigg\langle (v)^*, hT^n \frac{h}{T^n} v \bigg\rangle = \bigg\langle (v)^*, \left[2nk + \frac{h}{T^n} hT^n\right] v \bigg\rangle = 2 n k.
\end{gather*}
We have
\begin{gather*}
\bigg\langle \frac{h}{T^{\ell}} \left(\frac{f}{T^{a-1-\ell}}v\right)^*,
\frac{f}{T^{a-1-n}}\frac{h}{T^n} v \bigg\rangle = \bigg\langle \left(\frac{f}{T^{a-1-\ell}}v\right)^*, hT^{\ell} \frac{f}{T^{a-1-n}}\frac{h}{T^n} v \bigg\rangle,\\
hT^{\ell} \frac{f}{T^{a-1-n}}\frac{h}{T^n} v = \left[ -2 fT^{\ell+n-a+1} + \frac{f}{T^{a-1-n}} hT^{\ell} \right] \frac{h}{T^n} v \\
\hphantom{hT^{\ell} \frac{f}{T^{a-1-n}}\frac{h}{T^n} v} = -2 fT^{\ell+n-a+1} \frac{h}{T^n} v + \frac{f}{T^{a-1-n}} hT^{\ell} \frac{h}{T^n} v.
\end{gather*}
Note that the second summand is nonzero if and only if $\ell=n$. In that case we have
\begin{gather*}
\bigg\langle \left(\frac{f}{T^{a-1-n}}v\right)^*, \frac{f}{T^{a-1-n}} hT^n \frac{h}{T^n} v \bigg\rangle = 2nk.
\end{gather*}
For the first summand, if $\ell+n-a+1 \leq 0$, then $-2 fT^{\ell+n-a+1} \frac{h}{T^n} v$ is a basis vector and so pairing with $\left(\frac{f}{T^{a-1-\ell}}v\right)^*$ gives zero. If $\ell+n-a+1 > 0$, then
\begin{gather*}
\bigg\langle \left(\frac{f}{T^{a-1-\ell}}v\right)^*, -2 fT^{\ell+n-a+1} \frac{h}{T^n} v \bigg\rangle \\
\qquad {} = \bigg\langle \left(\frac{f}{T^{a-1-\ell}}v\right)^*, -2 \left[ 2\frac{f}{T^{a-1-\ell}} + \frac{h}{T^n} fT^{\ell+n-a+1}\right] v \bigg\rangle = - 4,
\end{gather*}
where we used $fT^{\ell+n-a+1}v = 0$.

Finally,
\begin{gather*}
\bigg\langle \frac{e}{T^{\ell}} \mathop{\sum_{i+j=a-1-\ell}}_{i\geq j\geq 0} \left(\frac{f}{T^i}\frac{f}{T^j}v\right)^*, \frac{f}{T^{a-1-n}}\frac{h}{T^n} v \bigg\rangle = \bigg\langle \mathop{\sum_{i+j=a-1-\ell}}_{i\geq j\geq 0} \left(\frac{f}{T^i}\frac{f}{T^j}v\right)^*, fT^{\ell} \frac{f}{T^{a-1-n}}\frac{h}{T^n} v \bigg\rangle, \\
fT^{\ell} \frac{f}{T^{a-1-n}}\frac{h}{T^n} v = \frac{f}{T^{a-1-n}} fT^{\ell} \frac{h}{T^n} v = \frac{f}{T^{a-1-n}} \left[ 2fT^{\ell-n} + \frac{h}{T^n} fT^{\ell} \right] v = 2 \frac{f}{T^{a-1-n}} \frac{f}{T^{n-\ell}} v,
\end{gather*}
since $ fT^{\ell} v = 0$. Note that $(a-1-n)+(n-\ell)=a-1-\ell$, hence if $i = a-1-n$ and $j = n-\ell$ (or vice versa depending on what is greater) we have
\begin{gather*}
\bigg\langle \mathop{\sum_{i+j=a-1-\ell}}_{i\geq j\geq 0} \left(\frac{f}{T^i}\frac{f}{T^j}v\right)^*, 2 \frac{f}{T^{a-1-n}} \frac{f}{T^{n-\ell}} v \bigg\rangle = 2,
\end{gather*}
whenever $n-\ell \geq 0$ and zero otherwise. The lemma is proved.
\end{proof}

For $s=1$ Proposition \ref{Pr3} follows from Lemma~\ref{Group 3, s=1}:
\begin{gather*}
\bigg\langle \frac{f}{T^{a-1}} (v)^* - \sum_{\ell=1}^{a-1} \bigg[\frac{h}{T^{\ell}}\left(\frac{f}{T^{a-1-\ell}}v\right)^* + 2\frac{e}{T^{\ell}} \mathop{\sum_{i+j=a-1-\ell}}_{i\geq j\geq 0} \left(\frac{f}{T^i}\frac{f}{T^j}v\right)^*\bigg], \frac{f}{T^{a-1-n}}\frac{h}{T^n} v \bigg\rangle \\
\qquad = 2nk - 2nk + 4n - 2 \cdot 2n = 0.
\end{gather*}

\begin{lem}\label{lem G1s}
For $s \geq 2$, we have
\begin{gather}
\bigg\langle \frac{f}{T^{a-1}}(v) ^*, w \bigg\rangle =0, \label{G1sg1} \\
\sum_{\ell=1}^{a-1} \bigg\langle \frac{h}{T^{\ell}} \left(\frac{f}{T^{a-1-\ell}}v\right)^*, w \bigg\rangle =0, \label{G1sg2} \\
\sum_{\ell=1}^{a-1} \bigg\langle \frac{e}{T^{\ell}} \mathop{\sum_{i+j=a-1-\ell}}_{i\geq j\geq 0} \left(\frac{f}{T^i}\frac{f}{T^j}v\right)^*, w \bigg\rangle = 0. \label{G1sg3}
\end{gather}
\end{lem}

\begin{proof}Recall that $w=\frac{f}{T^{a-1-n}}\frac{h}{T^{j_1}} \cdots \frac{h}{T^{j_s}} v$ with $j_1+\cdots+j_s=n$. We have
\begin{gather*}
\bigg\langle \frac{f}{T^{a-1}} (v)^*, w \bigg\rangle = \big\langle (v)^*, eT^{a-1} w \big\rangle, \\
eT^{a-1} w = eT^{a-1} \frac{f}{T^{a-1-n}}\frac{h}{T^{j_1}} \cdots \frac{h}{T^{j_s}} v = \left[ hT^n + \frac{f}{T^{a-1-n}} eT^{a-1} \right] \frac{h}{T^{j_1}} \cdots \frac{h}{T^{j_s}} v.
\end{gather*}
We have $hT^n \frac{h}{T^{j_1}} \cdots \frac{h}{T^{j_s}} v=0$, since $hT^n$ commutes with all $\frac{h}{T^{j_i}}$. Indeed, we have $n > j_i$ since $j_1 + \dots + j_s = n$, $j_i \geq 1$, and $s \geq 2$.

We also have $\frac{f}{T^{a-1-n}} eT^{a-1} \frac{h}{T^{j_1}} \cdots \frac{h}{T^{j_s}} v=0$ since $eT^{a-1} \frac{h}{T^{j_1}} \cdots \frac{h}{T^{j_s}}v$ is of degree $(-a+n, -a+1+n) \leq (-1, 0)$, hence zero. This proves~(\ref{G1sg1}).

We prove (\ref{G1sg2}) by induction on $s$. For $s=2$ we have
\begin{gather*}
\bigg\langle \frac{h}{T^{\ell}} \left(\frac{f}{T^{a-1-\ell}}v\right)^*, w \bigg\rangle = \bigg\langle \left(\frac{f}{T^{a-1-\ell}}v\right)^*, hT^{\ell} w \bigg\rangle, \\
hT^{\ell} w = hT^{\ell} \frac{f}{T^{a-1-n}}\frac{h}{T^{j_1}} \frac{h}{T^{j_2}} v = \left[ -2 fT^{\ell-a+1+n} + \frac{f}{T^{a-1-n}} hT^{\ell} \right] \frac{h}{T^{j_1}} \frac{h}{T^{j_2}} v \\
\hphantom{hT^{\ell} w}{} = -2 fT^{\ell-a+1+n} \frac{h}{T^{j_1}} \frac{h}{T^{j_2}} v + \frac{f}{T^{a-1-n}} hT^{\ell} \frac{h}{T^{j_1}} \frac{h}{T^{j_2}} v.
\end{gather*}
Note that for $ \frac{f}{T^{a-1-n}} hT^{\ell} \frac{h}{T^{j_1}} \frac{h}{T^{j_2}} v $ to give a nonzero pairing with $\big(\frac{f}{T^{a-1-\ell}}v\big)^*$ we need $\ell=n$, which implies that $hT^{\ell}$ commutes with $\frac{h}{T^{j_1}}$ and $\frac{h}{T^{j_2}}$ ($\ell > j_i$ since $j_1+j_2 =n=\ell$ and $j_i \geq 1$), so that $ \frac{f}{T^{a-1-n}} hT^{\ell} \frac{h}{T^{j_1}} \frac{h}{T^{j_2}} v $ gives zero for all $\ell$.

Also note that whenever $\ell \leq a-1-n$, $fT^{\ell-a+1+n} \frac{h}{T^{j_1}} \frac{h}{T^{j_2}} v$ is a basis vector and so pairing with $\big(\frac{f}{T^{a-1-\ell}}v\big)^*$ gives zero. If $\ell > a-1-n$, then
\begin{gather}
-2 fT^{\ell-a+1+n} \frac{h}{T^{j_1}} \frac{h}{T^{j_2}} v = -2 \left[ 2 fT^{\ell-a+1+n-j_1} + \frac{h}{T^{j_1}} fT^{\ell-a+1+n} \right] \frac{h}{T^{j_2}} v \nonumber \\
\hphantom{-2 fT^{\ell-a+1+n} \frac{h}{T^{j_1}} \frac{h}{T^{j_2}} v}{} = -4 fT^{\ell-a+1+n-j_1} \frac{h}{T^{j_2}} v -2 \frac{h}{T^{j_1}} fT^{\ell-a+1+n} \frac{h}{T^{j_2}} v. \label{1 and 2}
\end{gather}

If $\ell \leq a-1-n+j_1$, the first summand gives zero when pairing with $\big(\frac{f}{T^{a-1-\ell}}v\big)^*$, since for such $\ell fT^{\ell-a+1+n-j_1} \frac{h}{T^{j_2}} v$ is a basis vector. For $\ell > a-1-n+j_1$, we have
\begin{gather*}
-4 fT^{\ell-a+1+n-j_1} \frac{h}{T^{j_2}} v = -4 \left[ 2 \frac{f}{T^{a-1-\ell}} v + \frac{h}{T^{j_2}} fT^{\ell-a+1+n-j_1} \right] v= -8 \frac{f}{T^{a-1-\ell}} v,
\end{gather*}
since $fT^{\ell-a+1+n-j_1}v$ is of degree $(-\ell+a-1-n + j_1 +1, -\ell+a-1-n+j_1) \leq (0, -1)$, hence must be equal to zero. So for $\ell \in \{a-1-n+j_1 +1, \dots, a-1\}$ we get
\begin{gather*}
\bigg\langle \left(\frac{f}{T^{a-1-\ell}}v\right)^*, -4 fT^{\ell-a+1+n-j_1} \frac{h}{T^{j_2}} v \bigg\rangle = -8
\end{gather*}
and zero for other values of $\ell$. The total number of elements in the set $\{a-1-n+j_1 +1, \dots, a-1\}$ equals $j_2$.

Whereas, for the second summand in (\ref{1 and 2}) we have
\begin{gather*}
-2 \frac{h}{T^{j_1}} fT^{\ell-a+1+n} \frac{h}{T^{j_2}} v = -2\frac{h}{T^{j_1}} \left[ 2 fT^{\ell-a+1+n - j_2} + \frac{h}{T^{j_2}} fT^{\ell-a+1+n} \right] v \\
\hphantom{-2 \frac{h}{T^{j_1}} fT^{\ell-a+1+n} \frac{h}{T^{j_2}} v } = -4 \frac{h}{T^{j_1}} fT^{\ell-a+1+n -j_2} v,
\end{gather*}
since $fT^{\ell-a+1+n} v$ is of degree $(-\ell+a-1-n+1, -\ell+a-1-n) \leq (0, -1)$, hence must be equal to zero.

If $\ell > a-1-n+j_2$, then $fT^{\ell-a+1+n-j_2} v$ is of degree $(-\ell+a-1-n+j_2 +1, -\ell+a-1-n+j_2) \leq (0, -1)$, hence must be equal to zero. If $\ell \leq a-1-n+j_2$, then
\begin{gather*}
-4 \frac{h}{T^{j_1}} fT^{\ell-a+1+n -j_2} v = -4 \left[ -2\frac{f}{T^{a-1-\ell}} + fT^{\ell-a+1+n - j_2} \frac{h}{T^{j_1}} \right] \\
\hphantom{-4 \frac{h}{T^{j_1}} fT^{\ell-a+1+n -j_2} v} = 8 \frac{f}{T^{a-1-\ell}} -4 fT^{\ell-a+1+n - j_2} \frac{h}{T^{j_1}}.
\end{gather*}
The second summand gives zero when pairing with $\big(\frac{f}{T^{a-1-\ell}}v\big)^*$. So for $\ell \in \{a-1-n+1, \dots, a-1-n+j_2 \}$ we get
\begin{gather*}
\bigg\langle \left(\frac{f}{T^{a-1-\ell}}v\right)^*, -2 \frac{h}{T^{j_1}} fT^{\ell-a+1+n} \frac{h}{T^{j_2}} v \bigg\rangle = 8
\end{gather*}
and zero for other values of $\ell$. The total number of elements in the set $\{a-1-n+1, \dots, a-1-n+j_2 \}$ equals $j_2$.

Therefore,
\begin{gather*}
\sum_{\ell=1}^{a-1} \bigg\langle \frac{h}{T^{\ell}} \left(\frac{f}{T^{a-1-\ell}}v\right)^*, \frac{f}{T^{a-1-n}} \frac{h}{T^{j_1}} \frac{h}{T^{j_2}} v \bigg\rangle = -8 j_2 + 8 j_2 = 0
\end{gather*}
and so for $s=2$ we proved (\ref{G1sg2}).

Now suppose that (\ref{G1sg2}) holds for all natural numbers up to $s$. Then
\begin{gather*}
hT^{\ell} \frac{f}{T^{a-1-n}} \frac{h}{T^{j_1}}\dots\frac{h}{T^{j_{s+1}}} v = \left[ -2 fT^{\ell-a+1+n} + \frac{f}{T^{a-1-n}} hT^{\ell} \right] \frac{h}{T^{j_1}}\cdots\frac{h}{T^{j_{s+1}}} v.
\end{gather*}
Note that for $\frac{f}{T^{a-1-n}} hT^{\ell} \frac{h}{T^{j_1}} \cdots \frac{h}{T^{j_{s+1}}} v$ to give a nonzero pairing with $\big(\frac{f}{T^{a-1-\ell}}v\big)^*$
we need $\ell=n$. That assumption implies that $hT^{\ell}$ commutes with $\frac{h}{T^{j_i}}$ for all $i \in \{1, \dots, s+1\}$ since $\ell > j_i$ as $j_1+\dots +j_{s+1} =n=\ell$ and $j_i \geq 1$. Hence $ \frac{f}{T^{a-1-n}} hT^{\ell} \frac{h}{T^{j_1}}\cdots \frac{h}{T^{j_{s+1}}} v $ gives zero for all $\ell$.

Also note that whenever $\ell \leq a-1-n$, the vector $fT^{\ell-a+1+n} \frac{h}{T^{j_1}} \cdots \frac{h}{T^{j_{s+1}}} v$ is a basis vector and so pairing with $\big(\frac{f}{T^{a-1-\ell}}v\big)^*$ gives zero.

If $\ell > a-1-n$, then
\begin{gather}
-2 fT^{\ell-a+1+n} \frac{h}{T^{j_1}}\cdots\frac{h}{T^{j_{s+1}}} v=-2 \left[ 2fT^{\ell-a+1+n - j_1} + \frac{h}{T^{j_1}} fT^{\ell-a+1+n}\right] \frac{h}{T^{j_2}}\cdots\frac{h}{T^{j_{s+1}}} v \nonumber \\
\qquad{} = -4 \frac{f}{T^{a-1 - (n-j_1) - \ell}}\frac{h}{T^{j_2}}\cdots\frac{h}{T^{j_{s+1}}}v - 2 \frac{h}{T^{j_1}} fT^{-a+1 + n + \ell} \frac{h}{T^{j_2}}\cdots\frac{h}{T^{j_{s+1}}}v. \label{fs sum}
\end{gather}

Note that by induction hypothesis we have
\begin{gather*}
0 = \sum_{\ell=1}^{a-1} \bigg\langle \frac{h}{T^{\ell}} \left(\frac{f}{T^{a-1-\ell}}v\right)^*, \frac{f}{T^{a-1-(n-j_1)}} \frac{h}{T^{j_2}}\cdots\frac{h}{T^{j_{s+1}}} v \bigg\rangle \nonumber \\
\hphantom{0} = \sum_{\ell=1}^{a-1} \bigg\langle \left(\frac{f}{T^{a-1-\ell}}v\right)^*, hT^{\ell} \frac{f}{T^{a-1-(n-j_1)}} \frac{h}{T^{j_2}}\cdots\frac{h}{T^{j_{s+1}}} v \bigg\rangle. %\label{temporary}
\end{gather*}

So we add this zero term multiplied by $-2$ to the first summand in $(\ref{fs sum})$ to get
\begin{gather}
 -2 \sum_{\ell=1}^{a-1} \bigg\langle \left(\frac{f}{T^{a-1-\ell}}v\right)^*, 2\frac{f}{T^{a-1 - (n-j_1) -\ell}}\frac{h}{T^{j_2}}\cdots\frac{h}{T^{j_{s+1}}} v \bigg\rangle \nonumber \\
\qquad\quad{} -2 \sum_{\ell=1}^{a-1} \bigg\langle \left(\frac{f}{T^{a-1-\ell}}v\right)^*, hT^{\ell} \frac{f}{T^{a-1-(n-j_1)}} \frac{h}{T^{j_2}}\cdots\frac{h}{T^{j_{s+1}}} v \bigg\rangle \label{trick_notes} \\
\qquad{} = -2 \sum_{\ell=1}^{a-1} \bigg\langle \left(\frac{f}{T^{a-1-\ell}}v\right)^*, \left[ 2\frac{f}{T^{a-1 - (n-j_1) - \ell}} + hT^{\ell} \frac{f}{T^{a-1-(n-j_1)}} \right] \frac{h}{T^{j_2}}\dots\frac{h}{T^{j_{s+1}}} v \bigg\rangle \nonumber \\
\qquad{} = -2 \sum_{\ell=1}^{a-1} \bigg\langle \left(\frac{f}{T^{a-1-\ell}}v\right)^*, \frac{f}{T^{a-1-(n-j_1)}} hT^{\ell} \frac{h}{T^{j_2}}\cdots\frac{h}{T^{j_{s+1}}} v \bigg\rangle, \nonumber
\end{gather}
where in the last step we use commutation relations.

Note that for $ \frac{f}{T^{a-1-(n-j_1)}} hT^{\ell} \frac{h}{T^{j_2}} \cdots \frac{h}{T^{j_{s+1}}} v $ to give a nonzero pairing with $\big(\frac{f}{T^{a-1-\ell}}v\big)^*$ we need $\ell=n-j_1$. That assumption implies that $hT^{\ell}$ commutes with $\frac{h}{T^{j_i}}$ for all $i \in \{2, \dots, s+1\}$ since $\ell>j_i$ as $j_2+\dots+ j_{s+1} =n - j_1=\ell$ and $j_i \geq 1$. Hence $\frac{f}{T^{a-1-(n-j_1)}} hT^{\ell} \frac{h}{T^{j_2}} \cdots \frac{h}{T^{j_{s+1}}} v $ gives zero for all~$\ell$.

For the second summand in (\ref{fs sum}) we have
\begin{gather}
- 2 \sum_{\ell=1}^{a-1} \bigg\langle \left(\frac{f}{T^{a-1-\ell}}v\right)^*, \frac{h}{T^{j_1}} fT^{-a+1 + n + \ell} \frac{h}{T^{j_2}}\cdots\frac{h}{T^{j_{s+1}}}v \bigg\rangle
\nonumber \\
\qquad{} = - 2 \sum_{\ell=1}^{a-1} \bigg\langle \left(\frac{f}{T^{a-1-\ell}}v\right)^*, \frac{h}{T^{j_1}} \left[ 2 \frac{f}{T^{a-1 - (n-j_2) - \ell}} + \frac{h}{T^{j_2}} fT^{-a+1 + n +\ell} \right] \frac{h}{T^{j_3}} \cdots\frac{h}{T^{j_{s+1}}}v \bigg\rangle \nonumber \\
\qquad{} = - 4 \sum_{\ell=1}^{a-1} \bigg\langle \left(\frac{f}{T^{a-1-\ell}}v\right)^*, \frac{h}{T^{j_1}} \frac{f}{T^{a-1 - (n-j_2) - \ell}} \frac{h}{T^{j_3}} \cdots\frac{h}{T^{j_{s+1}}}v \bigg\rangle \label{ind1} \\
\qquad\quad{} - 2 \sum_{\ell=1}^{a-1} \bigg\langle \left(\frac{f}{T^{a-1-\ell}}v\right)^*, \frac{h}{T^{j_1}} \frac{h}{T^{j_2}} fT^{-a+1 + n +\ell} \frac{h}{T^{j_3}} \cdots\frac{h}{T^{j_{s+1}}}v \bigg\rangle. \label{ind2}
\end{gather}

In (\ref{ind1}) we note that
\begin{gather*}
\frac{h}{T^{j_1}} \frac{f}{T^{a-1 - (n-j_2) - \ell}} \frac{h}{T^{j_3}} \cdots\frac{h}{T^{j_{s+1}}}v \\
\qquad{} = \left[ -2 \frac{f}{T^{a-1 - (n-j_1-j_2) - \ell}} + \frac{f}{T^{a-1 - (n-j_2) - \ell}}\frac{h}{T^{j_1}} \right] \frac{h}{T^{j_3}} \cdots\frac{h}{T^{j_{s+1}}}v \\
\qquad{} = -2 \frac{f}{T^{a-1 - (n-j_1-j_2) - \ell}} \frac{h}{T^{j_3}} \cdots\frac{h}{T^{j_{s+1}}}v + \frac{f}{T^{a-1 - (n-j_2) - \ell}}\frac{h}{T^{j_1}} \frac{h}{T^{j_3}} \cdots\frac{h}{T^{j_{s+1}}}v.
\end{gather*}
Note that in both terms the number of $h$'s is less than or equal to $s$, so we use the exact same reasoning as in (\ref{trick_notes}) to show that
\begin{gather*}
\sum_{\ell=1}^{a-1} \bigg\langle \left(\frac{f}{T^{a-1-\ell}}v\right)^*, \frac{f}{T^{a-1 - (n-j_1-j_2) - \ell}} \frac{h}{T^{j_3}} \dots\frac{h}{T^{j_{s+1}}}v \bigg\rangle = 0, \\
\sum_{\ell=1}^{a-1} \bigg\langle \left(\frac{f}{T^{a-1-\ell}}v\right)^*, \frac{f}{T^{a-1 - (n-j_2) - \ell}}\frac{h}{T^{j_1}} \frac{h}{T^{j_3}} \cdots\frac{h}{T^{j_{s+1}}}v \bigg\rangle = 0,
\end{gather*}
which implies that the expression in (\ref{ind1}) equals zero. Similarly, one shows that in (\ref{ind2}),
\begin{gather*}
\sum_{\ell=1}^{a-1} \bigg\langle \left(\frac{f}{T^{a-1-\ell}}v\right)^*, \frac{h}{T^{j_1}} \frac{h}{T^{j_2}} fT^{-a+1 + n + \ell} \frac{h}{T^{j_3}} \cdots\frac{h}{T^{j_{s+1}}}v \bigg\rangle
\end{gather*}
the factor $fT^{-a+1 + n + \ell}$ can be pulled to the right by using the same argument (first commute $fT^{-a+1 + n + \ell}$ with $\frac{h}{T^{j_3}}$ and then pull $fT^{-a+1 + n-j_3 + \ell}$ to the left). Ultimately, we get
\begin{gather*}
\sum_{\ell=1}^{a-1} \bigg\langle \left(\frac{f}{T^{a-1-\ell}}v\right)^*, \frac{h}{T^{j_1}} \frac{h}{T^{j_2}} \frac{h}{T^{j_3}} \cdots\frac{h}{T^{j_{s+1}}} fT^{-a+1 + n + \ell} v \bigg\rangle = 0,
\end{gather*}
since $\ell > a-1-n$ and so $fT^{-a+1 + n + \ell} v = 0$. Therefore,
\begin{gather*}
\sum_{\ell=1}^{a-1} \bigg\langle \frac{h}{T^{\ell}} \left(\frac{f}{T^{a-1-\ell}}v\right)^*, \frac{f}{T^{a-1-n}} \frac{h}{T^{j_1}}\cdots\frac{h}{T^{j_{s}}} v \bigg\rangle = 0,
\end{gather*}
and formula (\ref{G1sg2}) is proved.

We prove formula (\ref{G1sg3}) by induction on $s$. For $s=2$, we have
\begin{gather*}
\bigg\langle \frac{e}{T^{\ell}} \mathop{\sum_{i+j=a-1-\ell}}_{i\geq j\geq 0} \left(\frac{f}{T^i}\frac{f}{T^j}v\right)^*, w \bigg\rangle = \bigg\langle \mathop{\sum_{i+j=a-1-\ell}}_{i\geq j\geq 0} \left(\frac{f}{T^i}\frac{f}{T^j}v\right)^*, fT^{\ell} w \bigg\rangle, \\
 fT^{\ell} w = fT^{\ell} \frac{f}{T^{a-1-n}}\frac{h}{T^{j_1}} \frac{h}{T^{j_2}} v = \frac{f}{T^{a-1-n}} fT^{\ell} \frac{h}{T^{j_1}} \frac{h}{T^{j_2}} v.
\end{gather*}
Note that $fT^{\ell} \frac{h}{T^{j_1}} \frac{h}{T^{j_2}} v$ is of degree $(n-\ell+1, n-\ell)$, hence nonzero only if $\ell \leq n$. For such $\ell$ we have
\begin{gather}
\frac{f}{T^{a-1-n}} \left[ 2fT^{\ell-j_1} + \frac{h}{T^{j_1}} ft^{\ell} \right] \frac{h}{T^{j_2}} v = 2 \frac{f}{T^{a-1-n}} fT^{\ell-j_1} \frac{h}{T^{j_2}} v + \frac{f}{T^{a-1-n}} \frac{h}{T^{j_1}} fT^{\ell} \frac{h}{T^{j_2}} v \nonumber \\
\hphantom{\frac{f}{T^{a-1-n}} \left[ 2fT^{\ell-j_1} + \frac{h}{T^{j_1}} ft^{\ell} \right] \frac{h}{T^{j_2}} v} = 2 \frac{f}{T^{a-1-n}} fT^{\ell-j_1} \frac{h}{T^{j_2}} v +2 \frac{f}{T^{a-1-n}} \frac{h}{T^{j_1}} fT^{\ell-j_2} v. \label{2terms}
\end{gather}
If $\ell \leq j_1$, then the first summand in~(\ref{2terms}) gives zero when pairing with any vector with two $f$'s. If $\ell > j_1$, then
\begin{gather*}
2 \frac{f}{T^{a-1-n}} fT^{\ell-j_1} \frac{h}{T^{j_2}} v = 2 \frac{f}{T^{a-1-n}} \left[ 2fT^{\ell-j_1-j_2} + \frac{h}{T^{j_2}}fT^{\ell-j_1} \right] v = 4 \frac{f}{T^{a-1-n}} \frac{f}{T^{n-\ell}} v.
\end{gather*}
If $\ell > j_2$, then the second summand in (\ref{2terms}) is zero simply because $fT^{\ell-j_2} v = 0$. If $\ell \leq j_2$, then
\begin{gather*}
2 \frac{f}{T^{a-1-n}} \frac{h}{T^{j_1}} fT^{\ell-j_2} v = 2 \frac{f}{T^{a-1-n}} \left[ -2 fT^{\ell-j_1-j_2} + fT^{\ell-j_2}\frac{h}{T^{j_1}} \right] v = -4 \frac{f}{T^{a-1-n}} \frac{f}{T^{n-\ell}} v,
\end{gather*}
since $\frac{f}{T^{a-1-n}} fT^{\ell-j_2}\frac{h}{T^{j_1}}v$ is a basis vector, hence pairing with a vector consisting of two $f$'s gives zero. Therefore,
\begin{gather}
\sum_{\ell=1}^{a-1} \bigg\langle \mathop{\sum_{i+j=a-1-\ell}}_{i\geq j\geq 0} \left(\frac{f}{T^i}\frac{f}{T^j}v\right)^*, fT^{\ell} w \bigg\rangle \nonumber \\
\qquad{} = \sum_{\ell=1}^{a-1} \bigg\langle \mathop{\sum_{i+j=a-1-\ell}}_{i\geq j\geq 0}
\left(\frac{f}{T^i}\frac{f}{T^j}v\right)^*, 4 \frac{f}{T^{a-1-n}} \frac{f}{T^{n-\ell}} v \bigg\rangle \label{sum1} \\
\qquad\quad{} + \sum_{\ell=1}^{a-1} \bigg\langle \mathop{\sum_{i+j=a-1-\ell}}_{i\geq j\geq 0} \left(\frac{f}{T^i}\frac{f}{T^j}v\right)^*, -4 \frac{f}{T^{a-1-n}} \frac{f}{T^{n-\ell}} v \bigg\rangle. \label{sum2}
\end{gather}
Note that in the expression in~(\ref{sum1}) for each $\ell \in \{j_1 +1, \dots, n \}$ there exists exactly one pair of indices $(i,j) = (\max \{a-1-n, n-\ell \} , \min \{a-1-n, n-\ell\})$ that gives $4$ when pairing. All other pairs $(i,j)$ give zero. Similarly, the expression in~(\ref{sum2}) equals $-4$ for each $\ell \in \{1, \dots, j_2 \}$ and exactly one corresponding pair $(i,j)$, and zero otherwise. Also note that the number of elements in each set $\{j_1 +1, \dots, n \}$ and $\{1, \dots, j_2 \}$ equals $j_2$. Hence we get
\begin{gather*}
4 j_2 - 4 j_2 = 0.
\end{gather*}
Therefore, formula (\ref{G1sg3}) is proved for $s=2$.

Now suppose that formula (\ref{G1sg3}) holds for all natural numbers up to $s$. Then
\begin{gather}
fT^{\ell} \frac{f}{T^{a-1-n}} \frac{h}{T^{j_1}} \cdots \frac{h}{T^{j_{s+1}}} v\ =\
 \frac{f}{T^{a-1-n}} fT^{\ell} \frac{h}{T^{j_1}} \cdots \frac{h}{T^{j_{s+1}}} v \nonumber \\
\qquad{} = \frac{f}{T^{a-1-n}} \left[ 2 fT^{\ell-j_1} + \frac{h}{T^{j_1}} fT^{\ell} \right] \frac{h}{T^{j_2}} \cdots \frac{h}{T^{j_{s+1}}} v \nonumber \\
\qquad{} = 2 \frac{f}{T^{a-1-n}} fT^{\ell-j_1} \frac{h}{T^{j_2}} \cdots \frac{h}{T^{j_{s+1}}} v + \frac{f}{T^{a-1-n}} \frac{h}{T^{j_1}} fT^{\ell} \frac{h}{T^{j_2}} \cdots \frac{h}{T^{j_{s+1}}} v. \label{fs sum2}
\end{gather}

Note that if $\ell \leq j_1$, then the first summand in (\ref{fs sum2}) is a basis vector and hence its pairing with a vector consisting of two $f$'s gives zero. If $\ell > j_1$ we have
\begin{gather*}
\sum_{\ell=1}^{a-1} \bigg\langle \mathop{\sum_{i+j=a-1-\ell}}_{i\geq j\geq 0} \left(\frac{f}{T^i}\frac{f}{T^j}v\right)^*, \frac{f}{T^{a-1-n}} fT^{\ell-j_1} \frac{h}{T^{j_2}} \cdots \frac{h}{T^{j_{s+1}}} v \bigg\rangle\\
\qquad{} = \sum_{\ell=j_1+1}^{a-1} \bigg\langle \mathop{\sum_{i+j=a-1-\ell}}_{i\geq j\geq 0} \left(\frac{f}{T^i}\frac{f}{T^j}v\right)^*, \frac{f}{T^{a-1-n}} fT^{\ell-j_1} \frac{h}{T^{j_2}} \cdots \frac{h}{T^{j_{s+1}}} v \bigg\rangle \\
\qquad{} = \sum_{\ell=j_1+1}^{a-1} \bigg\langle \frac{e}{T^{\ell-j_1}} \mathop{\sum_{i+j=a-1-\ell}}_{i\geq j\geq 0} \left(\frac{f}{T^i}\frac{f}{T^j}v\right)^*, \frac{f}{T^{a-1-n}} \frac{h}{T^{j_2}} \cdots \frac{h}{T^{j_{s+1}}} v \bigg\rangle \\
\qquad{} = \sum_{k=1}^{a-1-j_1} \bigg\langle \frac{e}{T^{k}} \mathop{\sum_{i+j=a-1-j_1 - k}}_{i\geq j\geq 0} \left(\frac{f}{T^i}\frac{f}{T^j}v\right)^*, \frac{f}{T^{(a-1-j_1)-(n-j_1)}} \frac{h}{T^{j_2}} \cdots \frac{h}{T^{j_{s+1}}} v \bigg\rangle = 0
\end{gather*}
by induction hypothesis. For the second summand in (\ref{fs sum2}) we have
\begin{gather}
\frac{f}{T^{a-1-n}} \frac{h}{T^{j_1}} fT^{\ell} \frac{h}{T^{j_2}} \cdots \frac{h}{T^{j_{s+1}}} v = \frac{f}{T^{a-1-n}} \frac{h}{T^{j_1}} \left[ 2fT^{\ell-j_2} + \frac{h}{T^{j_2}} fT^{\ell} \right] \frac{h}{T^{j_3}} \cdots \frac{h}{T^{j_{s+1}}} v \nonumber \\
\qquad{} = 2 \frac{f}{T^{a-1-n}} \frac{h}{T^{j_1}} fT^{\ell-j_2} \frac{h}{T^{j_3}} \cdots \frac{h}{T^{j_{s+1}}} v + \frac{f}{T^{a-1-n}} \frac{h}{T^{j_1}} \frac{h}{T^{j_2}} fT^{\ell} \frac{h}{T^{j_3}} \cdots \frac{h}{T^{j_{s+1}}} v. \label{2f}
\end{gather}
Note that
\begin{gather*}
\frac{f}{T^{a-1-n}} \frac{h}{T^{j_1}} fT^{\ell-j_2} \frac{h}{T^{j_3}} \cdots \frac{h}{T^{j_{s+1}}} v = \frac{f}{T^{a-1-n}} \left[ -2 fT^{\ell-j_1-j_2} + fT^{\ell-j_2} \frac{h}{T^{j_1}} \right] \frac{h}{T^{j_3}} \cdots \frac{h}{T^{j_{s+1}}} v \\
\qquad{} = -2 \frac{f}{T^{a-1-n}} fT^{\ell-j_1-j_2} \frac{h}{T^{j_3}} \cdots \frac{h}{T^{j_{s+1}}} v + \frac{f}{T^{a-1-n}} fT^{\ell-j_2} \frac{h}{T^{j_1}} \frac{h}{T^{j_3}} \cdots \frac{h}{T^{j_{s+1}}} v,
\end{gather*}
where in each vector the number of $h$'s is less than or equal to~$s$. Repeating the argument above, we see that by induction hypothesis we get
\begin{gather*}
\sum_{\ell=1}^{a-1} \bigg\langle \mathop{\sum_{i+j=a-1-\ell}}_{i\geq j\geq 0} \left(\frac{f}{T^i}\frac{f}{T^j}v\right)^*, \frac{f}{T^{a-1-n}} \frac{h}{T^{j_1}} fT^{\ell-j_2} \frac{h}{T^{j_3}} \cdots \frac{h}{T^{j_{s+1}}} v \bigg\rangle = 0.
\end{gather*}
Now in the second summand in (\ref{2f}),
\begin{gather*}
\frac{f}{T^{a-1-n}} \frac{h}{T^{j_1}} \frac{h}{T^{j_2}} fT^{\ell} \frac{h}{T^{j_3}} \cdots \frac{h}{T^{j_{s+1}}} v,
\end{gather*}
we pull $fT^{\ell}$ to the right and at each step we use induction hypothesis to argue that we keep getting zeros. Ultimately, we get a vector
\begin{gather*}
\frac{f}{T^{a-1-n}} \frac{h}{T^{j_1}} \cdots \frac{h}{T^{j_{s+1}}} fT^{\ell} v,
\end{gather*}
which is zero, since $fT^{\ell}$ has grading $(-\ell+1, -\ell) \leq (0,-1)$ and so $fT^{\ell}v=0$. Therefore,
\begin{gather*}
\sum_{\ell=1}^{a-1} \bigg\langle \frac{e}{T^{\ell}} \mathop{\sum_{i+j=a-1-\ell}}_{i\geq j\geq 0} \left(\frac{f}{T^i}\frac{f}{T^j}v\right)^*, \frac{f}{T^{a-1-n}} \frac{h}{T^{j_1}} \cdots \frac{h}{T^{j_{s}}} v \bigg\rangle = 0.
\end{gather*}
Formula (\ref{G1sg3}) and Lemma \ref{lem G1s} are proved.
\end{proof}
Proposition \ref{Pr3} is proved.
\end{proof}

\subsection[Group ${\rm II}$]{Group $\boldsymbol{{\rm II}}$}\label{subsec group 2}
\begin{prop}\label{Pr2}The value on the right-hand side of \eqref{AA} on any basis vector from Group~${\rm II}$ equals zero.
\end{prop}
\begin{proof}Group ${\rm II}$ consists of vectors
\begin{gather*}
w=\frac{f}{T^{i_1}}\frac{f}{T^{i_2}}\frac{h}{T^{j_1}} \cdots \frac{h}{T^{j_s}} \frac{e}{T^l}v.
\end{gather*}

\begin{lem}\label{Group 2} We have
\begin{gather}
\bigg\langle \frac{f}{T^{a-1}} (v)^*, w \bigg\rangle = 2^{s+1} (m - lk), \label{G21} \\
\bigg\langle \frac{h}{T^{\ell}} \left(\frac{f}{T^{a-1-\ell}}v\right)^*, w \bigg\rangle =
 \begin{cases}
 2^{s+2} (m - lk), & \text{if $i_1 = i_2 = a-1-\ell$},\\
 2^{s+1} (m - lk), & \text{if $i_1 \ne i_2$ and $i_{1}$ or $i_2 = a-1-\ell$}, \\
 0, & \text{otherwise},
 \end{cases} \label{G22} \\
\bigg\langle \frac{e}{T^{\ell}}
\mathop{\sum_{i+j=a-1-\ell}}_{i\geq j\geq 0} \left(\frac{f}{T^i}\frac{f}{T^j}v\right)^*, w \bigg\rangle =
 \begin{cases}
 - 2^{s} (m - lk), & \text{if $\ell = a-1 - i_1 - i_2$},\\
 0, & \text{otherwise}.
 \end{cases} \label{G23}
\end{gather}
\end{lem}

\begin{proof}We have
\begin{gather}
\bigg\langle \frac{f}{T^{a-1}} (v)^*, w \bigg\rangle = \big\langle (v)^*, eT^{a-1} w \big\rangle = \bigg\langle (v)^*, \left[ hT^{a-1-i_1} +
\frac{f}{T^{i_1}} eT^{a-1} \right]
 \frac{f}{T^{i_2}}\frac{h}{T^{j_1}} \cdots \frac{h}{T^{j_s}} \frac{e}{T^l}v \bigg\rangle \nonumber \\
{}= \bigg\langle (v)^*, hT^{a-1-i_1}
\frac{f}{T^{i_2}}\frac{h}{T^{j_1}} \cdots \frac{h}{T^{j_s}} \frac{e}{T^l}v \bigg\rangle +
\bigg\langle (v)^*, \frac{f}{T^{i_1}} eT^{a-1}
 \frac{f}{T^{i_2}}\frac{h}{T^{j_1}} \cdots \frac{h}{T^{j_s}} \frac{e}{T^l}v \bigg\rangle. \label{he}
\end{gather}
Note that $eT^{a-1} \frac{f}{T^{i_2}} \frac{h}{T^{j_1}} \cdots \frac{h}{T^{j_s}} \frac{e}{T^l}$ is of degree $(-i_1 - 1, -i_1) \leq (-1,0)$, hence \newline $eT^{a-1} \frac{f}{T^{i_2}} \frac{h}{T^{j_1}}\cdots\frac{h}{T^{j_s}} \frac{e}{T^{l}} v = 0$. In the first summand in~(\ref{he}) we pull $hT^{a-1-i_1}$ to the right to get
\begin{gather*}
\bigg\langle (v)^*, -2 fT^{a-1-i_1-i_2} \frac{h}{T^{j_1}} \cdots \frac{h}{T^{j_s}} \frac{e}{T^l}v \bigg\rangle = \dots = \bigg\langle (v)^*, -2^{s+1} fT^{a-1-i_1-i_2-j_1-\dots -a_s} \frac{e}{T^l}v \bigg\rangle \\
\qquad{} = \bigg\langle (v)^*, -2^{s+1} fT^{l} \frac{e}{T^l}v \bigg\rangle = \big\langle (v)^*, 2^{s+1} (h - lc) v \big\rangle = 2^{s+1} (m - lk),
\end{gather*}
where at each step we do not write monomials of negative degree, since they give zero when applied to~$v$. This proves formula~(\ref{G21}).

We have
\begin{gather*}
\bigg\langle \frac{h}{T^\ell} \left(\frac{f}{T^{a-1-\ell}}v\right)^*, w \bigg\rangle = \bigg\langle \left(\frac{f}{T^{a-1-\ell}}v\right)^*, hT^{\ell} w \bigg\rangle,\\
hT^{\ell} w = \left[ -2fT^{\ell-i_1} + \frac{f}{T^{i_1}} hT^{\ell} \right] \frac{f}{T^{i_2}}\frac{h}{T^{j_1}} \cdots \frac{h}{T^{j_s}} \frac{e}{T^l} v \\
\hphantom{hT^{\ell} w} = -2fT^{\ell-i_1} \frac{f}{T^{i_2}}\frac{h}{T^{j_1}} \cdots \frac{h}{T^{j_s}} \frac{e}{T^l} v + \frac{f}{T^{i_1}} \left[ -2fT^{\ell-i_2} + \frac{f}{T^{i_2}} hT^{\ell} \right] \frac{h}{T^{j_1}} \cdots \frac{h}{T^{j_s}} \frac{e}{T^l} v \\
\hphantom{hT^{\ell} w} = \left[ -2 \frac{f}{T^{i_2}} fT^{\ell-i_1} - 2 \frac{f}{T^{i_1}}fT^{\ell-i_2} + \frac{f}{T^{i_1}} \frac{f}{T^{i_2}} hT^{\ell} \right] \frac{h}{T^{j_1}} \cdots \frac{h}{T^{j_s}} \frac{e}{T^l} v.
\end{gather*}
Note that the vector $\frac{f}{T^{i_1}} \frac{f}{T^{i_2}} hT^{\ell} \frac{h}{T^{j_1}} \cdots \frac{h}{T^{j_s}} \frac{e}{T^l} v$ after pulling $hT^{\ell}$ to the right either becomes a~zero vector or a~vector with two $f$'s, which of course gives zero when pairing with a basis vector with one $f$. Also, note that the only possibility for the vector $\frac{f}{T^{i_2}} fT^{\ell-i_1} \frac{h}{T^{j_1}} \cdots \frac{h}{T^{j_s}} \frac{e}{T^l} v$ to give a~nonzero pairing with $\big(\frac{f}{T^{a-1-\ell}}v\big)^*$ is when $i_2 = a-1-\ell$. Similarly $\frac{f}{T^{i_1}} fT^{\ell-i_2} \frac{h}{T^{j_1}} \cdots \frac{h}{T^{j_s}} \frac{e}{T^l} v$ gives a nonzero number only if $i_1 = a-1-\ell$. First consider the case $i_1=i_2=a-1-\ell$. We have
\begin{gather*}
- 4 \frac{f}{T^{a-1-\ell}} fT^{2\ell-a+1} \frac{h}{T^{j_1}} \cdots \frac{h}{T^{j_s}} \frac{e}{T^l} v \\
\qquad{} = - 4 \frac{f}{T^{a-1-\ell}} \left[ 2 fT^{2\ell-a+1-j_1} + \frac{h}{T^{j_1}} fT^{2\ell-a+1} \right] \frac{h}{T^{j_2}} \cdots \frac{h}{T^{j_s}} \frac{e}{T^l} v.
\end{gather*}
Note that $fT^{2\ell-a+1} \frac{h}{T^{j_2}} \cdots \frac{h}{T^{j_s}} \frac{e}{T^l}$ is of degree $(-j_1,-j_1) \leq (-1,-1)$, hence
\begin{gather*} fT^{2\ell-a+1} \frac{h}{T^{j_2}} \cdots \frac{h}{T^{j_s}} \frac{e}{T^l} v =0.\end{gather*} So we get
\begin{gather*}
 -8 \frac{f}{T^{a-1-\ell}} fT^{2\ell-a+1-j_1} \frac{h}{T^{j_2}} \cdots \frac{h}{T^{j_s}} \frac{e}{T^l} v \cdots
 = -2^{s+2} \frac{f}{T^{a-1-\ell}} fT^{2\ell-a+1-j_1-\cdots-j_s} \frac{e}{T^l} v \\
\qquad{} = - 2^{s+2} \frac{f}{T^{a-1-\ell}} fT^{l} \frac{e}{T^l} v = 2^{s+2} \frac{f}{T^{a-1-\ell}} (h - lc) v = 2^{s+2} (m - lk),
\end{gather*}
where at each step we don't write monomials of negative degree, since they give zero when applied to $v$.

For $i_1 = a-1-\ell \ne i_2$ and $i_2 = a-1-\ell \ne i_1$ we have
\begin{gather*}
- 2 \frac{f}{T^{a-1-\ell}} fT^{2\ell-a+1} \frac{h}{T^{j_1}} \dots \frac{h}{T^{j_s}} \frac{e}{T^l} v = 2^{s+1} (m - lk),
\end{gather*}
where we performed the exact same computation as above. Therefore, formula (\ref{G22}) is proved.

We have
\begin{gather*}
 \bigg\langle \frac{e}{T^{\ell}} \mathop{\sum_{i+j=a-1-\ell}}_{i\geq j\geq 0} \left(\frac{f}{T^i}\frac{f}{T^j}v\right)^*, w \bigg\rangle = \bigg\langle \mathop{\sum_{i+j=a-1-\ell}}_{i\geq j\geq 0} \left(\frac{f}{T^i}\frac{f}{T^j}v\right)^*, fT^{\ell} w \bigg\rangle, \\
fT^{\ell} w = \frac{f}{T^{i_1}}\frac{f}{T^{i_2}} fT^{\ell} \frac{h}{T^{j_1}} \cdots \frac{h}{T^{j_s}} \frac{e}{T^l} v = \frac{f}{T^{i_1}}\frac{f}{T^{i_2}} \left[ 2 fT^{\ell-j_1} +\frac{h}{T^{j_1}} fT^\ell \right] \frac{h}{T^{j_2}} \cdots \frac{h}{T^{j_s}} \frac{e}{T^l} v.
\end{gather*}
The only nonzero pairing happens when $\ell$ is such that $i_1+i_2 = a-1-\ell$. In that case $fT^{\ell} \frac{h}{T^{j_2}} \cdots \frac{h}{T^{j_s}} \frac{e}{T^l}$ has degree $(-j_1, -j_1) \leq (-1,-1)$, hence $fT^{\ell} \frac{h}{T^{j_2}} \cdots \frac{h}{T^{j_s}} \frac{e}{T^l}v=0$. Therefore we have
\begin{gather*}
2 \frac{f}{T^{i_1}}\frac{f}{T^{i_2}} fT^{\ell-j_1} \frac{h}{T^{j_2}} \cdots \frac{h}{T^{j_s}} \frac{e}{T^l} v = 2^s \frac{f}{T^{i_1}}\frac{f}{T^{i_2}} fT^{\ell-j_1 - \dots - j_s} \frac{e}{T^l} v = 2^s \frac{f}{T^{i_1}}\frac{f}{T^{i_2}} fT^{l} \frac{e}{T^l} v,
\end{gather*}
where we pulled $fT^{\ell-j_1}$ to the right and did not write monomials of negative degree, since they give zero when applied to $v$. Hence we get
\begin{gather*}
2^s \frac{f}{T^{i_1}}\frac{f}{T^{i_2}} (-h +lc) v.
\end{gather*}
Therefore,
\begin{gather*}
\bigg\langle \frac{e}{T^{\ell}} \mathop{\sum_{i+j=a-1-\ell}}_{i\geq j\geq 0} \left(\frac{f}{T^i}\frac{f}{T^j}v\right)^*, w \bigg\rangle \\
\qquad{} = \mathop{\sum_{i+j=a-1-\ell}}_{i\geq j\geq 0}\bigg\langle \left(\frac{f}{T^i}\frac{f}{T^j}v\right)^*, 2^s \frac{f}{T^{i_1}}\frac{f}{T^{i_2}} (-h +lc) v \bigg\rangle = - 2^{s} (m - lk),
\end{gather*}
since for $i = i_1$, $j = i_2$ we get $-2^s (m - lk)$ and zero for other pairs $(i,j)$. Formula (\ref{G23}) and Lemma~\ref{Group 2} are proved.
\end{proof}

By Lemma \ref{Group 2}, we have
\begin{gather*}
\bigg\langle \frac{f}{T^{a-1}} (v)^*-
\sum_{\ell=1}^{a-1}\bigg[\frac{h}{T^{\ell}}\left(
\frac{f}{T^{a-1-\ell}}v\right)^* + 2\frac{e}{T^{\ell}} \mathop{\sum_{i+j=a-1-\ell}}_{i\geq j\geq 0} \left(\frac{f}{T^i}\frac{f}{T^j}v\right)^*\bigg], w \bigg\rangle \\
\qquad{} = 2^{s+1} (m-lk) - 2 \cdot 2^{s+1} (m-lk) + 2 \cdot 2^{s} (m-lk) = 0.
\end{gather*}
Note that
\begin{gather*}
 \bigg\langle \sum_{\ell=1}^{a-1} \frac{h}{T^{\ell}}\left(\frac{f}{T^{a-1-\ell}}v\right)^*, w \bigg\rangle = 2 \cdot 2^{s+1} (m - lk)
\end{gather*}
in both cases $i_1 = i_2$ and $i_1 \ne i_2$. Also note that
\begin{gather*}
\bigg\langle \sum_{\ell=1}^{a-1}\ \frac{e}{T^{\ell}} \mathop{\sum_{i+j=a-1-\ell}}_{i\geq j\geq 0} \left(\frac{f}{T^i}\frac{f}{T^j}v\right)^*, w \bigg\rangle \ne 0
\end{gather*}
only if $\ell$ is such that $i_1+i_2 = a-1-\ell$. Therefore, Proposition \ref{Pr2} is proved.
\end{proof}

\subsection[Group ${\rm III}$]{Group $\boldsymbol{{\rm III}}$}
\begin{prop}\label{Pr1}The value of the right-hand side of \eqref{AA} on any basis vector of Group~${\rm III}$ equals zero.
\end{prop}
\begin{proof}A vector in Group ${\rm III}$ has the form
\begin{gather*}
w= \frac{f}{T^{i_1}}\cdots\frac{f}{T^{i_r}}
\frac{h}{T^{j_1}}\cdots\frac{h}{T^{j_s}}
\frac{e}{T^{l_1}}\cdots\frac{e}{T^{l_{r-1}}}v,
\end{gather*}
where $r \geq 3$.

\begin{lem}\label{Group 1}For every $\ell \in \{1, \dots, a-1 \}$, we have
\begin{gather}
\bigg\langle \frac{f}{T^{a-1}} (v)^*, w \bigg\rangle =0, \label{G31} \\
\bigg\langle \frac{h}{T^{\ell}} \left(\frac{f}{T^{a-1-\ell}}v\right)^*, w \bigg\rangle =0, \label{G32} \\
 \bigg\langle \frac{e}{T^{\ell}} \mathop{\sum_{i+j=a-1-\ell}}_{i\geq j\geq 0} \left(\frac{f}{T^i}\frac{f}{T^j}v\right)^*, w \bigg\rangle = 0. \label{G33}
\end{gather}
\end{lem}

\begin{proof}We have
\begin{gather*}
 \bigg\langle \frac{f}{T^{a-1}} (v)^*, w \bigg\rangle = \big\langle (v)^*, eT^{a-1} w \big\rangle \\
\hphantom{\bigg\langle \frac{f}{T^{a-1}} (v)^*, w \bigg\rangle} = \bigg\langle (v)^*, \left[hT^{a-1-i_1} +\frac{f}{T^{i_1}} eT^{a-1}\right]
 \frac{f}{T^{i_2}} \cdots \frac{f}{T^{i_r}} \frac{h}{T^{j_1}}\cdots\frac{h}{T^{j_s}} \frac{e}{T^{l_1}}\cdots\frac{e}{T^{l_{r-1}}} v \bigg\rangle.
\end{gather*}
Note that $eT^{a-1} \frac{f}{T^{i_2}} \cdots \frac{f}{T^{i_r}} \frac{h}{T^{j_1}}\cdots\frac{h}{T^{j_s}} \frac{e}{T^{l_1}}\cdots\frac{e}{T^{l_{r-1}}} v $ is of degree $(-i_1 - 1,-i_1) \leq (-1,0)$, hence zero. So we have
\begin{gather*}
\bigg\langle (v)^*, \left[-2 fT^{a-1-i_1-i_2} +\frac{f}{T^{i_2}} hT^{a-1-i_1} \right]
\frac{f}{T^{i_3}} \cdots \frac{f}{T^{i_r}} \frac{h}{T^{j_1}}\cdots
\frac{h}{T^{j_s}} \frac{e}{T^{l_1}}\cdots\frac{e}{T^{l_{r-1}}} v \bigg\rangle.
\end{gather*}
As above, note that $hT^{a-1-i_1} \frac{f}{T^{i_3}} \cdots \frac{f}{T^{i_r}} \frac{h}{T^{j_1}}\cdots\frac{h}{T^{j_s}} \frac{e}{T^{l_1}}\cdots\frac{e}{T^{l_{r-1}}} v $ is of degree $ (-i_2-1, -i_2) \leq (-1,0)$, hence zero. Therefore we obtain
\begin{gather*}
\bigg\langle (v)^*, -2 \frac{f}{T^{i_3}} \cdots \frac{f}{T^{i_r}} fT^{a-1-i_1-i_2} \frac{h}{T^{j_1}}\cdots\frac{h}{T^{j_s}} \frac{e}{T^{l_1}}\cdots\frac{e}{T^{l_{r-1}}} v \bigg\rangle = 0,
\end{gather*}
since $fT^{a-1-i_1-i_2} \frac{h}{T^{j_1}}\cdots\frac{h}{T^{j_s}} \frac{e}{T^{l_1}}\cdots\frac{e}{T^{l_{r-1}}} v $ is of degree $(-i_3 - \dots - i_r - r + 2, -i_2 - \dots - i_r) \leq (-1,0)$ for $r \geq 3$, hence zero. Formula (\ref{G31}) is proved.

We have
\begin{gather}
 \bigg\langle \frac{h}{T^\ell} \left(\frac{f}{T^{a-1-\ell}}v\right)^*, w \bigg\rangle = \bigg\langle \left(\frac{f}{T^{a-1-\ell}}v\right)^*, hT^{\ell} w \bigg\rangle,\nonumber\\
hT^{\ell} w = \left[-2 fT^{\ell-i_1}+ \frac{f}{T^{i_1}} hT^{\ell} \right] \frac{f}{T^{i_2}}\cdots\frac{f}{T^{i_r}} \frac{h}{T^{j_1}}\cdots\frac{h}{T^{j_s}} \frac{e}{T^{l_1}}\cdots\frac{e}{T^{l_{r-1}}} v \nonumber \\
\hphantom{hT^{\ell} w} = -2 \frac{f}{T^{i_2}}\cdots\frac{f}{T^{i_r}} fT^{\ell-i_1} \frac{h}{T^{j_1}}\cdots\frac{h}{T^{j_s}} \frac{e}{T^{l_1}}\cdots\frac{e}{T^{l_{r-1}}} v \label{fr1} \\
\hphantom{hT^{\ell} w=} + \frac{f}{T^{i_1}} hT^{\ell} \frac{f}{T^{i_2}}\cdots\frac{f}{T^{i_r}} \frac{h}{T^{j_1}}\cdots\frac{h}{T^{j_s}} \frac{e}{T^{l_1}}\cdots\frac{e}{T^{l_{r-1}}} v. \label{fr2}
\end{gather}
Observe that for $\ell \leq i_1$ in~(\ref{fr1}) we have a basis vector, hence it gives zero when pairing with $\big(\frac{f}{T^{a-1-\ell}}v\big)^*$. If $\ell > i_1$, then we pull $fT^{\ell-i_1}$ to the right and notice that no matter how $fT^{\ell-i_1}$ interacts with $h$'s and $e$'s, it does not affect the number of~$f$'s, which is greater or equal than two. Hence, the vector in~(\ref{fr1}) gives zero when pairing with $\big(\frac{f}{T^{a-1-\ell}}v\big)^*$.

In (\ref{fr2}) note that
\begin{gather*}
hT^{\ell} \frac{f}{T^{i_2}} = -2 fT^{\ell-i_2} + \frac{f}{T^{i_2}} hT^{\ell}
\end{gather*}
so that either $hT^{\ell}$ is pulled to the right not affecting the number of $f$'s or it gives $fT^{\ell-i_2}$, for which we apply the same argument as above after pulling it to the right to argue that the pairing of the vector in~(\ref{fr2}) with $\big(\frac{f}{T^{a-1-\ell}}v\big)^*$ is zero. Formula (\ref{G32}) is proved.

We have
\begin{gather*}
\bigg\langle \frac{e}{T^{\ell}} \mathop{\sum_{i+j=a-1-\ell}}_{i\geq j\geq 0} \left(\frac{f}{T^i}\frac{f}{T^j}v\right)^*, w \bigg\rangle = \bigg\langle \mathop{\sum_{i+j=a-1-\ell}}_{i\geq j\geq 0} \left(\frac{f}{T^i}\frac{f}{T^j}v\right)^*, fT^{\ell} w \bigg\rangle, \\
fT^{\ell} w = fT^{\ell} \frac{f}{T^{i_1}}\cdots\frac{f}{T^{i_r}}
\frac{h}{T^{j_1}}\cdots\frac{h}{T^{j_s}} \frac{e}{T^{l_1}}\cdots\frac{e}{T^{l_{r-1}}} v \\
\hphantom{fT^{\ell} w} = \frac{f}{T^{i_1}}\cdots\frac{f}{T^{i_r}} fT^{\ell}
\frac{h}{T^{j_1}}\cdots\frac{h}{T^{j_s}} \frac{e}{T^{l_1}}\cdots\frac{e}{T^{l_{r-1}}} v.
\end{gather*}
As in formula (\ref{G32}), no matter how $fT^{\ell}$ interacts with $h$'s and $e$'s, the number of $f$'s remains unchanged, i.e., we have more than or equal to three~$f$'s, so that pairing with $\big(\frac{f}{T^i}\frac{f}{T^j}v\big)^*$ is zero. Formula~(\ref{G33}) is proved.
\end{proof}

Proposition \ref{Pr1} follows from Lemma \ref{Group 1}.
\end{proof}

Theorem \ref{thm form} is proved.

\subsection*{Acknowledgements}
The authors thank V.~Schechtman for useful discussions. The second author was supported in part by NSF grant DMS-1665239.

\pdfbookmark[1]{References}{ref}
\LastPageEnding

\end{document}